\documentclass{lmcs} 
\pdfoutput=1

\usepackage{lastpage}
\lmcsdoi{15}{1}{35}
\lmcsheading{}{\pageref{LastPage}}{}{}%
{Jan.~31,~2018}{Mar.~29,~2019}{}

\keywords{Moschovakis extension, represented space, effective topological space, effective metric space, TTE, partial function, multi-valued function, realization, abstract first order computability, computable, absolutely prime computable, combinatory space, iteration, first recursion theorem}

\usepackage{amssymb,upgreek,footmisc,xcolor,hyperref,bm}
\theoremstyle{plain}
\newtheorem{re}[thm]{Remark}
\newtheorem*{nre}{Remark}
\newtheorem{ex}[thm]{Example}
\theoremstyle{plain}\newtheorem{pr}[thm]{Proposition}
\newtheorem{co}[thm]{Corollary}
\newcommand{\dotminus}{\stackrel{\cdot}{\relbar}}
\def\eg{{\em e.g.}}
\def\cf{{\em cf.}}
\def\ie{{\em i.e.}}

\begin{document}
\title{Moschovakis extension of represented spaces}
\author{Dimiter Skordev}	
\address{Sofia University, Faculty of Mathematics and Informatics, Sofia, Bulgaria}	
\email{skordev@fmi.uni-sofia.bg}  





\begin{abstract}
\noindent Given a represented space (in the sense of TTE theory), an appropriate representation is constructed for the Moschovakis extension of its carrier (with paying attention to the cases of effective topological spaces and effective metric spaces). Some results are presented about TTE computability in the represented space obtained in this way. For single-valued functions, we prove, roughly speaking, the computability of any function which is absolutely prime computable in some computable functions. A similar result holds for multi-valued functions, but with an analog of absolute prime computability. The formulation of this result makes use of the notion of computability in iterative combinatory spaces -- a notion studied by the author in other publications.
\end{abstract}
\maketitle
\section*{Introduction}
By \cite[Theorems 3.1.6 and 3.1.7]{weihrauch:ca}, TTE computability of functions is preserved under composition and primitive recursion. In the case of composition, multi-valued functions are also admitted, but their composition is different from the usual one ({\ie} from the one reducible to composition of relations). In \cite{weihrauch:cmfmrscp} such closedness is proved with respect to flowchart programs in the general case of multi-valued functions in sets with generalized multi-representations. There are reasons to expect that all natural kinds of relative computability via deterministic programs using given single-valued functions can be obtained as instances of the notion of absolute prime computability introduced in \cite{moschovakis:afoc}. The notion in question makes sense also in the case of multi-valued given functions, but ceases in this case to be adequate for such a kind of study concerning TTE computability (this is due, in particular, to the fact that the usual composition of multi-valued functions does not preserve TTE computability\footnote{The TTE computability of multi-valued functions considered here is of the kind from \cite{weihrauch:ca}. It is referred to as weak one in \cite{brattka:cts}, where another computability notion for such functions is studied (see Definition~7.1 there).}). Fortunately, the general notion of computability in iterative combinatory spaces introduced by the present author ({\cf} {\eg} \cite{skordev:ccs}) embraces as particular instances numerous kinds of computability via programs (including absolute prime computability), and there is one among them which is appropriate for the case of programming through given multi-valued TTE computable functions.

The definition of absolute prime computability given in \cite{moschovakis:afoc} uses a construction which extends any given set by adding an additional object and building the closure with respect to ordered pairs. The extended set built in this way is used as a universe where the computations are carried. In Section \ref{S:1} of the present paper, two iterative combinatory spaces $\mathfrak{S}_1(X)$ and $\mathfrak{S}_2(X)$ are considered for any set $X$. Each of them is based on a semigroup of functions in the Moschovakis extension of $X$. In the case of $\mathfrak{S}_1(X)$, this is the semigroup of the unary single-valued partial functions with the composition operation as multiplication. In the case of $\mathfrak{S}_2(X)$, the semigroup consists of the unary multi-valued partial functions, and the other kind of composition mentioned above is used as multiplication. The computability in $\mathfrak{S}_1(X)$ is equivalent to absolute prime computability (restricted to single-valued partial functions). The computability in $\mathfrak{S}_2(X)$ is an analog of absolute prime computability, but it is different from it.
  
TTE computability is usually considered for functions between represented spaces. In Section~\ref{S:2}, a way is considered for turning into a represented space the Moschovakis extension $X^*$ of the carrier $X$ of any given represented space. Some particular cases are studied -- for instance, the case when the given represented space is the one generated by an effective topological space with carrier $X$ or the one generated by an effective metric space with carrier~$X$. In these cases, it is shown that the corresponding represented space with carrier $X^*$ can be generated by an effective space of the same kind with carrier~$X^*$.

In the last section, Section \ref{S:3}, the combinatory spaces $\mathfrak{S}_1(X)$ and $\mathfrak{S}_2(X)$ from Section~\ref{S:1} are considered in the case when the given set is the carrier $X$ of a represented space. It is shown that $\mathfrak{S}_1(X)$-computability and $\mathfrak{S}_2(X)$-computability preserve TTE computability in the case of single-valued functions and in the case of multi-valued ones, respectively. This allows applying the First Recursion Theorem from the general theory of iterative combinatory spaces for establishing TTE computability of some concrete functions. Two examples are given in which $X$ is the set of the real numbers and the usual computability of real numbers and functions is considered.

It is worth noting that the considerations in Sections \ref{S:2} and \ref{S:3} can be generalized for multi-represented spaces.

\begin{nre}
Since many definitions of TTE notions in different sources, although usually essentially equipollent, diverge in some details, the versions used in this paper of most definitions of such notions are explicitly indicated.
\end{nre}

\section{Moschovakis extensions and some combinatory spaces of functions in it}\label{S:1}

\subsection{The notion of Moschovakis extension}
Let $X$ be an arbitrary set. A \emph{Moschovakis extension} $X^*$ of $X$ is constructed by taking an object $\bm{o}$ not in $X$ and building the closure of $X\cup\{\bm{o}\}$ under the formation of ordered pairs, assuming the ordered pair operation is chosen in such a way that no element of $X\cup\{\bm{o}\}$ is an ordered pair (see~\cite{moschovakis:afoc}, where the set $X$ and the element $\bm{o}$ are denoted by $B$ and $0$, respectively).\footnote{The bold symbol $\bm{o}$ used here should not be confused with the normal text symbol $o$ used elsewhere by the author for least elements of iterative combinatory spaces.} Finite sequences of elements of $X^*$ can be encoded by elements of $X^*$ as follows: we use the element $\bm{o}$ as a code of the empty sequence, and, for any $k\in\mathbb{N}$, we encode the sequence $z_0,z_1,\ldots,z_k$ by the ordered pair $(z_0,z)$, where $z$ encodes the sequence $z_1,\ldots,z_k$ (for nonempty sequences, this encoding coincides with the one introduced in~\cite[(1.9)]{moschovakis:afoc}). For any $n\in\mathbb{N}$, let $n^*$ be the code of the $n$-term sequence of $\bm{o}$'s. The element $n^*$ will be regarded also as a code of the natural number $n$, thus the numbers $0,\,1,\,2,\,3,\,\ldots$ will be encoded by the elements $\bm{o},\,(\bm{o},\bm{o}),\,(\bm{o},(\bm{o},\bm{o})),\,(\bm{o},(\bm{o},(\bm{o},\bm{o}))),\,\ldots$ of~$X^*$.\footnote{Moschovakis uses, so to speak, the mirror-image of this encoding of the natural numbers -- in the present notation, the elements $\bm{o},\,(\bm{o},\bm{o}),\,((\bm{o},\bm{o}),\bm{o}),\,(((\bm{o},\bm{o}),\bm{o}),\bm{o}),\,\ldots$ are the ones used by him (\cf~\cite[(1.3)]{moschovakis:afoc}).} Let $\mathbb{N}^*=\{n^*\,|\,n\in\mathbb{N}\}$, where $\mathbb{N}$, as usual, is the set of the natural numbers. Having in mind the above or similar encodings, we may represent partial functions in $X^*$ of any number of arguments by unary partial functions in $X^*$, and partial functions of any number of arguments from $X^*$ to $\mathbb{N}$ by unary partial functions from the set $X^*$ to its subset~$\mathbb{N}^*$. Partial functions of any number of arguments from $\mathbb{N}$ to $X^*$ and such ones in $\mathbb{N}$ can be similarly represented by unary partial functions in~$X^*$.

\subsection{Combinatory spaces. The spaces \texorpdfstring{$\mathfrak{S}_1(X)$}{S1(X)} and \texorpdfstring{$\mathfrak{S}_2(X)$}{S2(X)}}

We will recall below the definition of the notion of combinatory space from \cite{skordev:ccs} (Definition~1 on pp.~45--46 there). For the sake of ease of comprehension, we, however, will first give an example of combinatory space (in this example which is an instance of Example~1 on p.~48 of \cite{skordev:ccs}, it would be appropriate to read $(x,y)\in\theta$ as follows: ``$y$ is an image of $x$ under $\theta$''). Let $X$ be a set, and let
\begin{equation}\label{cs}
\mathfrak{S}=(\mathcal{F},I,\mathcal{C},\Pi,L,R,\Sigma,T,F),
\end{equation}
where $\mathcal{F}$ is the set of all binary relations in $X^*$ considered with the composition operation
\[\varphi\psi=\{(x,z)\mid\exists y((x,y)\in\psi\ \&\ (y,z)\in\varphi)\}\]
and the partial ordering by inclusion, $I$ is the diagonal relation on $X^*$, $\mathcal{C}$ is the set of all relations of the form $X^*\times\{c\}$, where $c\in X^*$, $\Pi$ (juxtaposition) is the binary operation in $\mathcal{F}$ defined by
\[\Pi(\varphi,\psi)=\{(z,(x,y))\mid(z,x)\in\varphi\ \&\ (z,y)\in\psi\},\]
$L=\{((x,y),x)\mid x,y\in X^*\}$, $R=\{((x,y),y)\mid x,y\in X^*\}$, $\Sigma$ (branching) is the ternary operation in $\mathcal{F}$ defined by
\begin{multline*}
\Sigma(\chi,\varphi,\psi)=\{(x,z)\mid\exists y((x,y)\in\chi\\\&\ ((y\in X^*\setminus(X\cup\{\bm{o}\})\ \&\ (x,z)\in\varphi)\vee(y\in X\cup\{\bm{o}\}\ \&\ (x,z)\in\psi))\},
\end{multline*}
$T=X^*\times\{(\bm{o},\bm{o})\}$, $F=X^*\times\{\bm{o}\}$.
Then $\mathfrak{S}$ is a combinatory space in the sense of the above-mentioned definition which reads as follows (up to some changes in the notations): a combinatory space is a 9-tuple \eqref{cs},   
where $\mathcal{F}$ is a partially ordered semigroup, $I$ is an identity of this semigroup, $\mathcal{C}$ is a subset of~$\mathcal{F}$, $\Pi$ and $\Sigma$ are, respectively, a binary and a ternary operation in $\mathcal{F}$, $L,R,T,F$ are elements of $\mathcal{F}$, and the following conditions are identically satisfied when $\varphi,\psi,\theta,\chi$ range over $\mathcal{F}$ and $\xi,\eta$ range over $\mathcal{C}$:
{
\allowdisplaybreaks
\begin{gather}
\forall\xi(\varphi\xi\ge\psi\xi)\Rightarrow\varphi\ge\psi,\\
\Pi(\xi,\eta)\in\mathcal{C},\ \ L\Pi(\xi,\eta)=\xi,\ \ R\Pi(\xi,\eta)=\eta,\\
\Pi(\varphi,\psi)\xi=\Pi(\varphi\xi,\psi\xi),\\
\Pi(I,\psi\xi)\theta=\Pi(\theta,\psi\xi),\\
\Pi(\xi,I)\theta=\Pi(\xi,\theta),\label{five}\\
T\ne F,\ \ T\xi\in\mathcal{C},\ \ F\xi\in\mathcal{C},\\
\Sigma(T,\varphi,\psi)=\varphi,\ \ \Sigma(F,\varphi,\psi)=\psi,\\
\theta\Sigma(\chi,\varphi,\psi)=\Sigma(\chi,\theta\varphi,\theta\psi),\\
\Sigma(\chi,\varphi,\psi)\xi=\Sigma(\chi\xi,\varphi\xi,\psi\xi),\\
\Sigma(I,\varphi\xi,\psi\xi)\theta=\Sigma(\theta,\varphi\xi,\psi\xi),\\
\varphi\ge\psi\,\&\,\theta\ge\chi\Rightarrow\Sigma(I,\varphi,\theta)\ge\Sigma(I,\psi,\chi).
\end{gather}
}
It easily follows from the definition that not only the multiplication in $\mathcal{F}$ is monotonically increasing, but so are also the operations $\Pi$ and $\Sigma$.
 
Obviously, \eqref{five} is the particular instance with $\varphi=I$ of the following equality:
\[\Pi(\varphi\xi,I)\theta=\Pi(\varphi\xi,\theta).\]
The combinatory space $\mathfrak{S}$ is called {\em symmetric} if this equality holds for all $\varphi,\theta$ in $\mathcal{F}$ and all $\xi$ in $\mathcal{C}$ (the combinatory space from the above example is a symmetric one).\footnote{In publications of the author prior to \cite{skordev:ccs} (from 1975 on), a somewhat different notion of combinatory space is used. Roughly speaking, the combinatory spaces considered in  these publications are the symmetric combinatory spaces such that $T$ and $F$ belong to $\mathcal{C}$. Such kind of combinatory spaces actually would be sufficient for the present paper. We note, however, that some changes ought to be done in the description of the combinatory spaces $\mathfrak{S}_1(X)$ and $\mathfrak{S}_2(X)$ introduced next if we would want to be in concordance with the publications in question: one should exchange $\bm{o}$ and $(\bm{o},\bm{o})$ in the definition of $T_1$ and $F_1$, as well as $X\cup\{\bm{o}\}$ and $X^*\setminus(X\cup\{\bm{o}\})$ in the definition of $\Sigma_1$, and a similar change would be needed in the definition of $\Sigma_2$.} 

We will define now two different symmetric combinatory spaces \[\mathfrak{S}_1(X)=(\mathcal{F}_1,I_1,\mathcal{C}_1,\Pi_1,L_1,R_1,\Sigma_1,T_1,F_1),\ \ \mathfrak{S}_2(X)=(\mathcal{F}_2,I_2,\mathcal{C}_2,\Pi_2,L_2,R_2,\Sigma_2,T_2,F_2)\]
with semigroups $\mathcal{F}_1$ and $\mathcal{F}_2$ consisting of functions in the Moschovakis extension $X^*$ of a an arbitrary given set~$X$.

The combinatory space $\mathfrak{S}_1(X)$ is an instance of the ones considered in \cite[Example~2 on p.~48]{skordev:ccs}. We take as $\mathcal{F}_1$ the partially ordered semigroup of all partial mappings of $X^*$ into $X^*$ (including the total ones) -- with multiplication defined as composition, \ie
\[\varphi\psi(z)=\varphi(\psi(z)),\]
where $\mathrm{dom}(\varphi\psi)=\psi^{-1}(\mathrm{dom}(\varphi))$, and partial order defined by
\[\psi\ge\varphi\Leftrightarrow\psi\text{ is an extension of }\varphi.\] The identity mapping of $X^*$ onto $X^*$ is taken as $I_1$, and $L_1,\,R_1\in\mathcal{F}_1$ are defined in the following way:
\begin{align*}
&L_1(\bm{o})=R_1(\bm{o})=\bm{o},\\
&L_1(z)=R_1(z)=(\bm{o},\bm{o})\ \text{ if }z\in X,\\
&L_1(z)=x,\ R_1(z)=y\ \text{ if }z=(x,y)
\end{align*}
($L_1$ and $R_1$ are the functions $\pi$ and $\delta$ from ~\cite{moschovakis:afoc}). The set of all constant mappings of $X^*$ into $X^*$ is taken as $\mathcal{C}_1$, and $T_1,\,F_1$ are those of them whose values are $(\bm{o},\bm{o})$ and $\bm{o}$, respectively. The operations $\Pi_1$ and $\Sigma_1$ are defined in $\mathcal{F}_1$ as follows:
\begin{gather*}
\Pi_1(\varphi,\psi)(z)=(\varphi(z),\psi(z)),\\
\Sigma_1(\chi,\varphi,\psi)(x)=
\begin{cases}
\varphi(x)&\text{\!if }x\in\chi^{-1}(X^*\setminus(X\cup\{\bm{o}\})),\\
\psi(x)&\text{\!if }x\in\chi^{-1}(X\cup\{\bm{o}\}),
\end{cases}
\end{gather*}
where
\begin{gather*}
\mathrm{dom}(\Pi_1(\varphi,\psi))=\mathrm{dom}(\varphi)\cap\mathrm{dom}(\psi),\\
\mathrm{dom}(\Sigma_1(\chi,\varphi,\psi))=\left(\chi^{-1}(X^*\setminus(X\cup\{\bm{o}\}))\cap\mathrm{dom}(\varphi)\right)\cup\left(\chi^{-1}(X\cup\{\bm{o}\})\cap\mathrm{dom}(\psi)\right).
\end{gather*}

The second combinatory space can be regarded as an extension of the first one. The elements of $\mathcal{F}_2$ will be the partial functions from $X^*$ to the set of the non-empty subsets of~$X^*$, and any $\theta\in\mathcal{F}_1$ will have as its counterpart in $\mathcal{F}_2$ the function $\tilde{\theta}$ transforming each $x\in\mathrm{dom}(\theta)$ into the one-element set $\{\theta(x)\}$. The multiplication in $\mathcal{F}_2$ will be defined in the way indicated in the last sentence on p.~289 of \cite{skordev:ccs}, namely
\begin{gather*}
\mathrm{dom}(\varphi\psi)=\{x\in\mathrm{dom}(\psi)\,|\,\psi(x)\subseteq\mathrm{dom}(\varphi)\},\\
\forall x\in\mathrm{dom}(\varphi\psi)\left(\varphi\psi(x)=\bigcup\{\varphi(y)\,|\,y\in\psi(x)\}\right)
\end{gather*}
(compare with \cite[p.~11]{weihrauch:ca} and the second composition operation introduced in \cite[Definition~6]{weihrauch:cmfmrscp}). The partial order in $\mathcal{F}_2$ will be as in Remark~9 on p.~293 of \cite{skordev:ccs}:
\[\varphi\ge\psi\Leftrightarrow\mathrm{dom}(\varphi)\supseteq\mathrm{dom}(\psi)\,\&\,\forall x\in\mathrm{dom}(\psi)(\varphi(x)\subseteq\psi(x))\]
(compare with \cite[Definition~7]{weihrauch:cmfmrscp}). As one immediately sees, the above-mentioned mapping $\theta\mapsto\tilde{\theta}$ of the partially ordered semigroup $\mathcal{F}_1$ into the partially ordered semigroup $\mathcal{F}_2$ is an isomorphic embedding. We set further
\[\mathcal{C}_2=\{\tilde{\xi}\,|\,\xi\in\mathcal{C}_1\},\]
and we take as $L_2,R_2,T_2,F_2$ the images of $L_1,R_1,T_1,F_1$, respectively, under the mapping $\theta\mapsto\tilde{\theta}$. Finally we set
\begin{gather*}
\Pi_2(\varphi,\psi)(z)=\{(x,y)\,|\,x\in\varphi(z)\,\&\,y\in\psi(z)\},\\
\begin{split}
\Sigma_2(\chi,\varphi,\psi)&(x)\\=\{z\,|\,&(\chi(x)\setminus(X\cup\{\bm{o}\})\ne\varnothing\,\&\,z\in\varphi(x))
\vee(\chi(x)\cap(X\cup\{\bm{o}\})\ne\varnothing\,\&\,z\in\psi(x))\}
\end{split}
\end{gather*}
with
\begin{gather*}
\mathrm{dom}(\Pi_2(\varphi,\psi))=\mathrm{dom}(\varphi)\cap\mathrm{dom}(\psi),\\
\begin{split}
\mathrm{dom}&(\Sigma_2(\chi,\varphi,\psi))=\{x\in\mathrm{dom}(\chi)\,|\,\\
&(\chi(x)\setminus(X\cup\{\bm{o}\})\ne\varnothing\Rightarrow x\in\mathrm{dom}(\varphi))\ \&\ (\chi(x)\cap(X\cup\{\bm{o}\})\ne\varnothing\Rightarrow x\in\mathrm{dom}(\psi))\}
\end{split}
\end{gather*}
(\cf~\cite[the text on pp. 290--291 between Remark~2 and Remark~3, as well as Remark~6 on p.~292]{skordev:ccs}). One verifies that, for all $\chi,\varphi,\psi\in\mathcal{F}_1$, the elements $\Pi_2(\tilde{\varphi},\tilde{\psi})$ and $\Sigma_2(\tilde{\chi},\tilde{\varphi},\tilde{\psi})$ of $\mathcal{F}_2$ are the images of $\Pi_1(\varphi,\psi)$ and $\Sigma_1(\chi,\varphi,\psi)$, respectively, under the mapping $\theta\mapsto\tilde{\theta}$.
\subsection{Iterative combinatory spaces}
If a combinatory space \eqref{cs} is given, and $\sigma,\chi$ are elements of its component $\mathcal{F}$, then the notion of iteration of $\sigma$ controlled by $\chi$ is introduced. Its definition makes use of certain other notions. Let $\mathcal{A}$ be any subset of $\mathcal{C}$. By definition, two elements $\varphi$ and $\psi$ of $\mathcal{F}$ satisfy the inequality $\varphi\underset{\mathcal{A}}{\ge}\psi$ if $\forall\xi\in\mathcal{A}(\varphi\xi\ge\psi\xi)$. The set $\mathcal{A}$ is said to be {\em invariant} with respect to an element $\sigma$ of $\mathcal{F}$ if $\varphi\underset{\mathcal{A}}{\ge}\psi\Rightarrow\varphi\sigma\underset{\mathcal{A}}{\ge}\psi\sigma$ for all $\varphi,\psi\in\mathcal{F}$. By \cite[Definition~1 on p.~74]{skordev:ccs}, {\em the iteration of $\sigma$ controlled by $\chi$} is an element $\iota$ of $\mathcal{F}$ such that $\iota=\Sigma(\chi,\iota\sigma,I)$ and, for all $\tau,\rho\in\mathcal{F}$ and each subset $\mathcal{A}$ of $\mathcal{F}$ which is invariant with respect to $\sigma$, the inequality $\tau\underset{\mathcal{A}}{\ge}\Sigma(\chi,\tau\sigma,\rho)$ implies the inequality $\tau\underset{\mathcal{A}}{\ge}\rho\iota$.\footnote{In the author's publications before 1980, somewhat stronger requirements are imposed on the iteration. This difference, however, is not important for the considerations in the present paper.} By applying the last requirement with $\mathcal{A}=\mathcal{C}$ and $\rho=I$, one makes the following conclusion: if $\iota$ is the iteration of $\sigma$ controlled by $\chi$ then $\iota$ is the least $\tau$ such that $\tau\ge\Sigma(\chi,\tau\sigma,I)$. Thus if the iteration of $\sigma$ controlled by $\chi$ exists then it is a uniquely determined element of $\mathcal{F}$. As in \cite{skordev:ccs}, this element will be denoted by~$[\sigma,\chi]$.

The combinatory space \eqref{cs} is called {\em iterative} if $[\sigma,\chi]$ exists for all $\sigma,\chi\in\mathcal{F}$. It is easily seen that the operation of iteration in an iterative combinatory space is monotonically increasing.

For any given set $X$, the combinatory spaces $\mathfrak{S}_1(X)$ and $\mathfrak{S}_2(X)$ turn out to be iterative. For a class of combinatory spaces which contains $\mathfrak{S}_1(X)$, the property of being iterative is indicated on p.~74 of \cite{skordev:ccs}. For any $\sigma,\chi\in\mathcal{F}_1$, the partial function $[\sigma,\chi]$ can be characterized as follows:\footnote{Assuming that, for $u,u'\in X^*$ and $\varphi\in\mathcal{F}_1$, the condition $u\in\mathrm{dom}(\varphi)$ is taken for granted in statements of the form $\varphi(u)=u'$ and in ones of the form $u'=\varphi(u)$.} $[\sigma,\chi](z)=v$ iff $z$ and $v$ are the initial and the last term, respectively, of a finite sequence of elements of $\mathrm{dom}(\chi)$ such that
\[\chi(u)\not\in X\cup\{\bm{o}\},\ \ u'=\sigma(u),\]
whenever $u$ is a term of this sequence and $u'$ is its next term, and, additionally,
\[\chi(v)\in X\cup\{\bm{o}\}\]
(thus the function $\iota=[\sigma,\chi]$ can be represented in Pascal-like notations by means of the declaration
\begin{align*}
\text{fu}&\text{nction }\iota(z:X^*):X^*;\\
&\text{var }u:X^*;\\
\text{be}&\text{gin}\\
&u:=z;\\
&\text{while }\chi(u)\not\in X\cup\{\bm{o}\}\text{ do }u:=\sigma(u);\\
&\iota:=u\\
\text{en}&\text{d;}).
\end{align*}

The fact that $\mathfrak{S}_2(X)$ is iterative can be seen by using \cite[Proposition~2 on p.~291 and Remark~9 on p.~293]{skordev:ccs}.\footnote{Unfortunately, there are two misprints in condition~(a) of the mentioned proposition, namely the third and the fourth $M$ in this condition must be replaced with $Q$.} In this combinatory space, the characterization of $[\sigma,\chi]$ is more complicated. Actually the description of the set $[\sigma,\chi](z)$ for $z\in
\mathrm{dom}([\sigma,\chi])$ is a natural modification of the characterization of the value $[\sigma,\chi](z)$ in the case of $\mathfrak{S}_1(X)$ and looks as follows:\footnote{Assuming that, for $u,u'\in X^*$ and $\varphi\in\mathcal{F}_2$, the condition $u\in\mathrm{dom}(\varphi)$ is taken for granted in statements of the form $u'\in\varphi(u)$.} $v\in[\sigma,\chi](z)$ iff $z$ and $v$ are the initial and the last term, respectively, of a finite sequence of elements of $\mathrm{dom}(\chi)$ such that
\begin{equation}\label{pth}
\chi(u)\setminus(X\cup\{\bm{o}\})\ne\varnothing,\ \ u'\in\sigma(u),
\end{equation}
whenever $u$ is a term of this sequence and $u'$ is its next term, and, additionally, 
\[\chi(v)\cap(X\cup\{\bm{o}\})\ne\varnothing.\]
The complications occur in the description of $\mathrm{dom}([\sigma,\chi])$. Let us call a {\em $\sigma,\chi$-path} any sequence (finite or infinite) of elements of $X^*$ such that, whenever $u$ is a term of this sequence and $u'$ is its next term, then $u\in\mathrm{dom}(\chi)\cap\mathrm{dom}(\sigma)$ and the conditions \eqref{pth} are satisfied. Then (see \cite[Remark~4 on p.~292]{skordev:ccs}) $\mathrm{dom}([\sigma,\chi])$ consists of the elements $z$ of $X^*$ which have the following two properties:
\begin{enumerate}[label=(\roman*)]
\item all $\sigma,\chi$-paths beginning at $z$ are finite;
\item whenever $u$ is a term of some $\sigma,\chi$-path beginning at $z$, then
\begin{equation}\label{regular}
u\in\mathrm{dom}(\chi)\,\&\,(\chi(u)\setminus(X\cup\{\bm{o}\})\ne\varnothing\Rightarrow u\in\mathrm{dom}(\sigma))
\end{equation}
\end{enumerate}
(clearly, we may restrict ourselves in (ii) to the last terms of the $\sigma,\chi$-paths beginning at $z$).

\begin{ex}\label{branching} 
Let $\beta$ be the function from $\mathcal{F}_2$ with domain $X^*$ such that $\beta(z)=\{\bm{o},(\bm{o},\bm{o})\}$ for any $z\in X^*$, and let $\sigma,\chi\in\mathcal{F}_2$ be defined as follows:
\[\sigma=\Sigma_2(\beta,L_2,R_2),\ \ \chi=\Sigma_2(I_2,\beta,I_2).\]
Let $\kappa=[\sigma,\chi]$. Then $\mathrm{dom}(\kappa)=X^*$ and
\[\kappa(z)=
\begin{cases}
\{z\}&\text{if }z\in X\cup\{\bm{o}\},\\
\{z\}\cup\kappa(L_1(z))\cup\kappa(R_1(z))&\text{otherwise}.
\end{cases}
\]
\end{ex}

One verifies that, for all $\sigma,\chi\in\mathcal{F}_1$, the iteration in $\mathfrak{S}_2(X)$ of $\tilde{\sigma}$ controlled by $\tilde{\chi}$ is the image under the mapping $\theta\mapsto\tilde{\theta}$ of the iteration in $\mathfrak{S}_1(X)$ of $\sigma$ controlled by $\chi$.

\subsection{Computability in iterative combinatory spaces}
Let an iterative combinatory space \eqref{cs} be given. By \cite[Definition~1 on p.~142]{skordev:ccs}, an element $\theta$ of $\mathcal{F}$ is said to be {\em $\mathfrak{S}$-computable in a subset $\mathcal{B}$} of $\mathcal{F}$ if $\theta$ can be obtained from $L,\,R,\,T,\,F$ and finitely many elements of $\mathcal{B}$ by means of multiplication in the semigroup $\mathcal{F}$ and the operations $\Pi$ and iteration\footnote{In fact, $T$ and $F$ are superfluous in the case of $\mathfrak{S}=\mathfrak{S}_1$ or $\mathfrak{S}=\mathfrak{S}_2$, since we have the equalities $T=L^3[L,L]$, $F=\Pi(T,T)$ in this case. Besides this, whenever the combinatory space $\mathfrak{S}$ is iterative, the following equalities hold for all $\chi,\varphi,\psi\in\mathcal{F}$:
\begin{equation}\label{reduct}
I=[\varphi,F],\ \Sigma(\chi,\varphi,\psi)=R[\widehat{F}\psi R,L]R[\widehat{F}^2\varphi R^2,L]\Pi(\chi,\widehat{T})
\end{equation}
with $\widehat{T}=\Pi(T,I)$ and $\widehat{F}=\Pi(F,I)$ (\cf~\cite[Example~1 on p.~74, Proposition~1 on p.~103 and Proposition~4 on p.~104]{skordev:ccs}), thus $I$ is $\mathfrak{S}$-computable in $\varnothing$, and the operation $\Sigma$ preserves $\mathfrak{S}$-computability in $\mathcal{B}$.\label{sf}}. According to \cite[Definition~2 on p.~143]{skordev:ccs}, a~mapping $\Gamma$ of $\mathcal{F}^n$ into $\mathcal{F}$ is said to be {\em $\mathfrak{S}$-computable in $\mathcal{B}$} if, for $\theta_1,\ldots,\theta_n\in\mathcal{F}$, the element $\Gamma(\theta_1,\ldots,\theta_n)$ is $\mathfrak{S}$-computable in $\mathcal{B}\cup\{\theta_1,\ldots,\theta_n\}$ uniformly with respect to $\theta_1,\ldots,\theta_n$.\footnote{For instance (by the second of the equalities \eqref{reduct}), the mapping $\Sigma$  of $\mathcal{F}^3$ into $\mathcal{F}$ is $\mathfrak{S}$-computable in~$\varnothing$.\label{compSigma}} Of course, if $\mathcal{B}\subseteq\mathcal{B}'\subseteq\mathcal{F}$, then $\mathfrak{S}$-computability in $\mathcal{B}$ implies $\mathfrak{S}$-computability in $\mathcal{B}'$.

As before, $X$ is assumed to be an arbitrary given set.

\begin{ex}
Let us use the notations of Example \ref{branching}. As seen from this example, the function $\kappa$ is $\mathfrak{S}_2(X)$-computable in the set~$\{\beta\}$. The mapping $\Gamma:\mathcal{F}_2^2\to\mathcal{F}_2$ defined by
\[\Gamma(\theta_1,\theta_2)=\uplambda z.\theta_1(z)\cup\theta_2(z)\]
(with $\mathrm{dom}(\Gamma(\theta_1,\theta_2))=\mathrm{dom}(\theta_1)\cap\mathrm{dom}(\theta_2)$) is also $\mathfrak{S}_2(X)$-computable in the set~$\{\beta\}$. This follows from the fact that
\[\Gamma(\theta_1,\theta_2)=\Sigma_2(\beta,\theta_1,\theta_2)\]
for all $\theta_1,\theta_2\in\mathcal{F}_2$.
\end{ex}

\begin{re}
If $\theta\in\mathcal{F}_2$ then the graph of $\theta$ is the set $\{(z,u)\,|\,z\in\mathrm{dom}(\theta)\,\&\,u\in\theta(z)\}$. A mapping of $\mathcal{F}_2^2$ into $\mathcal{F}_2$ also related to unions is the one which transforms any pair of functions from $\mathcal{F}_2$ into the function whose graph is the union of their graphs. This mapping, however, is not monotonically increasing with respect to the partial order in $\mathfrak{S}_2(X)$, therefore it is $\mathfrak{S}_2(X)$-computable in no subset of $\mathcal{F}_2$.
\end{re}

If $\mathcal{B}_1$ is a subset of $\mathcal{F}_1$, and $\mathcal{B}_2$ is the image of $\mathcal{B}_1$ under the mapping $\theta\mapsto\tilde{\theta}$, it is easy to see that, for any element $\theta$ of $\mathcal{F}_1$, the $\mathfrak{S}_1(X)$-computability of $\theta$ in $\mathcal{B}_1$ is equivalent to the $\mathfrak{S}_2(X)$-computability of $\tilde{\theta}$ in $\mathcal{B}_2$.

From \cite[Theorem~1 on pp.~192--193]{skordev:ccs} it is seen that, for any subset $\mathcal{B}$ of $\mathcal{F}_1$, the $\mathfrak{S}_1(X)$-computability in $\mathcal{B}$ of an element $\theta$ of $\mathcal{F}_1$ is equivalent to absolute prime computability of the function $\theta$ in some finite subset of $\mathcal{B}$.\footnote{For the definition of absolute prime computability, \cf~\cite{moschovakis:afoc} (to apply this definition to a function $\theta$ from $\mathcal{F}_1$, one must actually consider the partial multiple-valued function which transforms any element $z$ of $X^*$ into the one-element set $\{\theta(z)\}$ if $z\in\mathrm{dom}(\theta)$ and into the empty set otherwise).} The proof of this theorem makes use of \cite[Theorem~1 on p.~170]{skordev:ccs} (the First Recursion Theorem for iterative combinatory spaces). According to it, if $\Gamma_1,\ldots,\Gamma_l$ are mappings of $\mathcal{F}^l$ into $\mathcal{F}$ which are $\mathfrak{S}$-computable in a subset $\mathcal{B}$ of $\mathcal{F}$ then the system of inequalities
\begin{equation}\label{frt}
\theta_r\ge\Gamma_r(\theta_1,\ldots,\theta_l),\ \ r=1,\ldots,l,
\end{equation}
has a least solution, and the components of this solution are $\mathfrak{S}$-computable in~$\mathcal{B}$ (actually, the mentioned inequalities are satisfied as equalities by this solution). A parameterized version \cite[Theorem~2 on pp.~170--171]{skordev:ccs} of the theorem in question concerns systems of the form
\[\theta_r\ge\Gamma_r(\theta_1,\ldots,\theta_l,\theta_{l+1},\ldots,\theta_{l+m}),\ \ r=1,\ldots,l,\]
where $\Gamma_1,\ldots,\Gamma_l$ are mappings of $\mathcal{F}^{l+m}$ into $\mathcal{F}$ which are $\mathfrak{S}$-computable in~$\mathcal{B}$. The existence of mappings $\Delta_1,\ldots,\Delta_l$ of $\mathcal{F}^m$ into $\mathcal{F}$ $\mathfrak{S}$-computable in~$\mathcal{B}$ is asserted such that, for any choice of $\theta_{l+1},\ldots,\theta_{l+m}$ in~$\mathcal{F}$, the $l$-tuple $(\Delta_1(\theta_{l+1},\ldots,\theta_{l+m}),\ldots,\Delta_l(\theta_{l+1},\ldots,\theta_{l+m})$ is the least solution of the above system with respect to $\theta_1,\ldots,\theta_l$ (the inequalities are satisfied as equalities again). 

Here is an example of an application to $\mathfrak{S}_1(X)$ of the First Recursion Theorem for iterative combinatory spaces.  

\begin{ex} 
Let $\Gamma$ be the mapping $\theta\mapsto\Sigma_1(I_1,\Pi_1(\theta R_1,\theta L_1),I_1)$ of $\mathcal{F}_1$ into $\mathcal{F}_1$. Clearly this mapping is $\mathfrak{S}_1(X)$-computable in $\varnothing$. By the definition of the operations in $\mathcal{F}$, we have
\[\Gamma(\theta)(z)=
\begin{cases}
z&\text{\!if }z\in X\cup\{\bm{o}\},\\
(\theta(R_1(z)),\theta(L_1(z)))&\text{\!otherwise}
\end{cases}\]
for all $\theta\in\mathcal{F}$ and all $z\in X^*$. Using this, it is easy to see directly the existence of a unique fixed point of $\Gamma$ (this fixed point is the total function in $X^*$ which transforms any element of $X^*$ into its ``mirror-image''). By the First Recursion Theorem for iterative combinatory spaces, the fixed point of $\Gamma$ is $\mathfrak{S}_1(X)$-computable in $\varnothing$ (hence absolutely prime computable in~$\varnothing$).
\end{ex}

To conclude that the least fixed point in the above example is absolutely prime computable in $\varnothing$, one could of course use also the First Recursion Theorem for prime computability (for the details concerning it, see \cite[p.~460]{moschovakis:afoc}). The notion of absolute prime computability makes sense also for multi-valued functions. However, no obvious way is seen for applying the First Recursion Theorem for prime computability to problems about $\mathfrak{S}_2(X)$-computability, because the $\mathfrak{S}_2(X)$-computability of a function from $\mathcal{F}_2$ in some finite subset $\mathcal{B}$ of $\mathcal{F}_2$ can turn out to be not equivalent to absolute prime computability in $\mathcal{B}$ (thus $\mathfrak{S}_2(X)$-computability is an analog to absolute prime computability, but it is different from it).\footnote{Of course, when applying the definition of absolute prime computability to a function $\theta$ from $\mathcal{F}_2$ one must actually consider the extension of $\theta$ as a partial multiple-valued function whose value is empty everywhere in $X^*\setminus\mathrm{dom}(\theta)$.}

\begin{ex}
Let $\beta$ be the same as in Example \ref{branching}. Making use of K\H{o}nig's Lemma, one sees that the values of the functions $\mathfrak{S}_2(X)$-computable in the set $\{\beta\}$ are always finite sets. On the other hand, the function transforming all elements of $X^*$ into the infinite set $\mathbb{N}^*$ is absolutely prime computable in $\beta$ -- this can be shown by using \cite[Theorem~1 on pp.~192--193]{skordev:ccs} and the equality $\mathbb{N}^*=\iota(\bm{o})$, where $\iota$ is the  iteration of the function $z\mapsto(\bm{o},z)$ controlled by $\beta$ in the simpler sense relevant for the situation.\footnote{If $\sigma,\chi\in\mathcal{F}_2$ then this kind of iteration of $\sigma$ controlled by $\chi$ is the function $\iota\in\mathcal{F}_2$ defined as follows: 
\begin{enumerate}[label=(\roman*)]
\item $\mathrm{dom}(\iota)$ consists of the initial terms of the finite $\sigma,\chi$-paths whose last term belongs to the set
\[\{v\in\mathrm{dom}(\chi)\,|\,\chi(v)\cap(X\cup\{\bm{o}\})\ne\varnothing\};\]
\item for any $z\in\mathrm{dom}(\iota)$, the set $\iota(z)$ consists of the last terms of those of the above-mentioned sequences whose initial term is $z$.
\end{enumerate}
This is the iteration of $\sigma$ controlled by $\chi$ in an iterative combinatory space not essentially different from $\mathfrak{S}_m(\mathfrak{M}_X)$, where $\mathfrak{S}_m(\mathfrak{M}_B)$ is the combinatory space considered in the proof of the mentioned Theorem~1.} Hence absolute prime computability in $\beta$ of a unary function in $X^*$ does not imply its $\mathfrak{S}_2(X)$-computability in~$\beta$. The converse implication does not hold either. To show this, let us note that the notion of recursive enumerability makes sense for subsets of the closure of $\{\bm{o}\}$ under the pairing operation. It can be shown that, whenever a partial multiple-valued unary function $\theta$ in $X^*$ is absolutely prime computable in $\beta$, the intersection of $\{z\in X^*\,|\,\theta(z)\ne\varnothing\}$ with the closure in question is recursively enumerable. On the other hand, an $\mathfrak{S}_2(X)$-computable in $\beta$ function $\theta$ can be constructed such that $\mathrm{dom}(\theta)$ is a subset of this closure without being recursively enumerable.
\end{ex}

\section{Moschovakis extension of a represented space}\label{S:2}

\subsection{A general construction}
As in \cite[Definition~2.1]{brattka:cts}, a {\em representation} of a set is a partial mapping of $\mathbb{N}^\mathbb{N}$ onto this set (where, of course, $\mathbb{N}^\mathbb{N}$ is the set of all functions from $\mathbb{N}$ to $\mathbb{N}$), and a {\em represented space} is an ordered pair whose terms are a set and a representation of this set. 

Let $(X,\varrho)$ be a represented space, and let a Moschovakis extension $X^*$ of $X$ be constructed in the way from Section \ref{S:1}. We choose a computable bijection $J$ from~$\mathbb{N}^2$ to $\mathbb{N}$ such that the inequality $J(m,n)\ge\max(m,n)$ holds for all $m,n\in\mathbb{N}$, and we define a partial mapping $\varrho^*$ of $\mathbb{N}^\mathbb{N}$ into $X^*$ in the following inductive way, where $k$ ranges over $\mathbb{N}$:
\begin{enumerate}[label=(\roman*)]
\item For any $q\in\mathrm{dom}(\varrho)$, $\uplambda k.2q(k)+2\in\mathrm{dom}(\varrho^*)$ and 
$\varrho^*(\uplambda k.2q(k)+2)=\varrho(q)$.
\item $\uplambda k.0\in\mathrm{dom}(\varrho^*)$ and $\varrho^*(\uplambda k.0)=\bm{o}$.
\item If $q,r\in\mathrm{dom}(\varrho^*)$ then $\uplambda k.2J(q(k),r(k))+1\in\mathrm{dom}(\varrho^*)$ and
\[\varrho^*(\uplambda k.2J(q(k),r(k))+1)=(\varrho^*(q),\varrho^*(r)).\]
\end{enumerate}
Clearly the range of the partial mapping $\varrho^*$ is the whole $X^*$, hence $(X^*,\varrho^*)$ is a represented space. We will regard it as a Moschovakis extension of $(X,\varrho)$.

If the mapping $\varrho$ is injective then so is $\varrho^*$. On the other hand, $\varrho^*$ is surely not total (even in the case when $\varrho$ is total). For instance, $\varrho^*$ is not defined for the elements of $\mathbb{N}^\mathbb{N}$ which have both zero and nonzero values or have both even and odd values.

Of course, the representation $\varrho^*$ depends on the choice of the computable bijection~$J$. However, if $J'$ is an arbitrary computable bijection from~$\mathbb{N}^2$ to $\mathbb{N}$ such that we have $J'(m,n)\ge\max(m,n)$ for all $m,n\in\mathbb{N}$ then the corresponding representation ${\varrho^*}'$ is connected with $\varrho^*$ in a simple way, namely a computable bijection $h$ from $\mathbb{N}$ to $\mathbb{N}$ exists such that $\mathrm{dom}({\varrho^*}')=h^{-1}(\mathrm{dom}(\varrho^*))$ and, for any $p\in\mathrm{dom}({\varrho^*}')$, the equality ${\varrho^*}'(p)=\varrho^*(hp)$ holds (where, of course, $hp=\uplambda k.h(p(k))$). In fact, we may define $h$ in the following inductive way:
\[h(2l)=2l,\ \ h(2J'(m,n)+1)=2J(h(m),h(n))+1.\]

In the next two subsections, we will indicate two cases when properties of a represented space $(X,\varrho)$ are preserved by the formation of its Moschovakis extension $(X^*,\varrho^*)$, namely the case when $(X,\varrho)$ is generated by some effective topological space and the case when $(X,\varrho)$ is generated by some effective metric space.

\subsection{The case of represented space generated by an effective topological space}
An effective topological space can be regarded as an ordered pair $(X,\mathcal{U})$, where $X$ is a set, and $\mathcal{U}=\{\mathcal{U}_i\}_{i\in\mathbb{N}}$ is a base for a T$_0$ topology in $X$. In this situation, let us set $\mathcal{U}^{-1}(x)=\{i\in\mathbb{N}\,|\,x\in\mathcal{U}_i\}$ for any $x\in X$. The represented space corresponding to $(X,\mathcal{U})$ has as its second term the partial mapping $\varrho$ of $\mathbb{N}^\mathbb{N}$ into $X$ defined as follows: $\varrho(p)=x$ iff $\mathrm{rng}(p)=\mathcal{U}^{-1}(x)$. Let us define a family $\mathcal{U}^*=\{\mathcal{U}^*_j\}_{j\in\mathbb{N}}$ of subsets of $X^*$ by means of the equalities
\[\ \mathcal{U}^*_0=\{\bm{o}\},\ \mathcal{U}^*_{2i+2}=\mathcal{U}_i,
\ \mathcal{U}^*_{2J(m,n)+1}=\mathcal{U}^*_m\times\mathcal{U}^*_n.\]

\begin{pr}
The ordered pair $(X^*,\mathcal{U}^*)$ is an effective topological space, and $(X^*,\varrho^*)$ is the represented space generated by it. In addition, if the effective topological space $(X,\mathcal{U})$ is computable then the effective topological space $(X^*,\mathcal{U}^*)$ is also computable.
\end{pr}

\begin{proof}
The statement that $(X^*,\mathcal{U}^*)$ is an effective topological space means that $\mathcal{U}^*$ is a base for a T$_0$ topology in $X^*$.

Let $Z_0=\bigcup_{i=0}^\infty\mathcal{U}^*_i$. We will show that $Z_0=X^*$ by showing that $X\cup\{\bm{o}\}\subseteq Z_0$ and $Z_0$ is closed under formation of ordered pairs. Indeed, if $z\in X$ then $z\in\mathcal{U}_i=\mathcal{U}^*_{2i+2}\subseteq Z_0$ for some $i$, the element $\bm{o}$ also belongs to $Z_0$ because $\bm{o}\in\mathcal{U}^*_0\subseteq Z_0$, and if $u,v\in Z_0$ then $u\in\mathcal{U}^*_m$ and $v\in\mathcal{U}^*_n$ for some $m$ and $n$, therefore $(u,v)\in\mathcal{U}^*_{2J(m,n)+1}\subseteq Z_0$. 

The next thing we have to show is the following statement: if $z\in\mathcal{U}^*_{i_1}\cap\mathcal{U}^*_{i_2}$ then there is some $i$ such that $z\in\mathcal{U}^*_i\subseteq\mathcal{U}^*_{i_1}\cap\mathcal{U}^*_{i_2}$. Let $Z_1$ be the set of all elements $z$ of $X^*$ such that the above implication holds for all $i_1,i_2$. We will show that $Z_1=X^*$ by showing that $X\cup\{\bm{o}\}\subseteq Z_1$ and $Z_1$ is closed under formation of ordered pairs. Indeed, if $z\in X$ and $z\in\mathcal{U}^*_{i_1}\cap\mathcal{U}^*_{i_2}$ then $\mathcal{U}^*_{i_1}=\mathcal{U}_{l_1}$, $\mathcal{U}^*_{i_2}=\mathcal{U}_{l_2}$ for some $l_1,l_2$, hence $z\in\mathcal{U}_{l_1}\cap\mathcal{U}_{l_2}$, therefore we have $z\in\mathcal{U}_l\subseteq\mathcal{U}_{l_1}\cap\mathcal{U}_{l_2}$ for some $l$, {\ie} $z\in\mathcal{U}^*_i\subseteq\mathcal{U}^*_{i_1}\cap\mathcal{U}^*_{i_2}$ with $i=2l+2$. If $\bm{o}\in\mathcal{U}^*_{i_1}\cap\mathcal{U}^*_{i_2}$ then $i_1=i_2=0$, hence $\bm{o}\in\mathcal{U}^*_i\subseteq\mathcal{U}^*_{i_1}\cap\mathcal{U}^*_{i_2}$ with $i=0$. Thus $X\cup\{\bm{o}\}\subseteq Z_1$. Suppose now that $u,v\in Z_1$ and $(u,v)\in\mathcal{U}^*_{i_1}\cap\mathcal{U}^*_{i_2}$. Then $\mathcal{U}^*_{i_1}=\mathcal{U}^*_{m_1}\times\mathcal{U}^*_{n_1}$, $\mathcal{U}^*_{i_2}=\mathcal{U}^*_{m_2}\times\mathcal{U}^*_{n_2}$ for some $m_1,n_1,m_2,n_2$. Clearly $u\in\mathcal{U}^*_{m_1}\cap\mathcal{U}^*_{m_2}$ and $v\in\mathcal{U}^*_{n_1}\cap\mathcal{U}^*_{n_2}$, therefore $u\in\mathcal{U}^*_m\subseteq\mathcal{U}^*_{m_1}\cap\mathcal{U}^*_{m_2}$ and $v\in\mathcal{U}^*_n\subseteq\mathcal{U}^*_{n_1}\cap\mathcal{U}^*_{n_2}$ for some $m,n$, hence $(u,v)\in\mathcal{U}^*_i\subseteq\mathcal{U}^*_{i_1}\cap\mathcal{U}^*_{i_2}$ with $i=2J(m,n)+1$.

Up to here, we showed that $\mathcal{U}^*$ is a base for some topology in $X^*$. We have to show that the topology in question is a T$_0$ one. We will achieve this goal by showing that any $z\in X^*$ has the following property: whenever $z'\in X^*$ and $z'\ne z$, then there is some $i$ such that $\mathcal{U}^*_i$ contains one of $z$ and $z'$ but not the other. Let $Z_2$ be the set of the elements $z$ of $X^*$ with this property. The elements of $X$ belong to $Z_2$. Indeed, suppose $z\in X$, $z'\in X^*$ and $z'\ne z$. If $z'\in X$, there is some $i$ such that $\mathcal{U}_i$ contains one of $z$ and $z'$ but not the other, and if $z'\not\in X$ then no $\mathcal{U}_i$ contains $z'$ and therefore some $\mathcal{U}_i$ contains $z$ but not $z'$. In both cases there is some $i$ such that $\mathcal{U}^*_{2i+2}$ contains one of $z$ and $z'$ but not the other. The element $\bm{o}$ also belongs to $Z_2$ because $\mathcal{U}^*_0$ contains $\bm{o}$, but does not contain other elements of $X^*$. So $X\cup\{\bm{o}\}\subseteq Z_2$. Suppose now $u,v\in Z_2$, $z'\in X^*$ and $z'\ne(u,v)$. If $z'\in X\cup\{\bm{o}\}$ then $z'\in Z_2$ and therefore there is some $i$ such that $\mathcal{U}^*_i$ contains one of $z'$ and $(u,v)$ but not the other.\footnote{Actually the stronger statement holds that some $\mathcal{U}^*_i$ contains $z'$ but does not contain $(u,v)$.} Let $z'\not\in X\cup\{\bm{o}\}$. Then $z'=(u',v')$ for some $u',v'\in X^*$, and some of the inequalities $u'\ne u$, $v'\ne v$ holds. Suppose $u'\ne u$. Then there is some $m$ such that $\mathcal{U}^*_m$ contains one of $u$ and $u'$ but not the other. If $u\in\mathcal{U}^*_m$ then we consider some $n$ such that $v\in\mathcal{U}^*_n$, and we observe that $\mathcal{U}^*_{2J(m,n)+1}$ contains $(u,v)$ but does not contain $z'$. If $u'\in\mathcal{U}^*_m$ then we consider some $n$ such that $v'\in\mathcal{U}^*_n$, and we observe that $\mathcal{U}^*_{2J(m,n)+1}$ contains $z'$ but does not contain $(u,v)$. The reasoning in the case $v'\ne v$ is similar.

So far, we have established that $(X^*,\mathcal{U}^*)$ is an effective topological space. As such one, it generates a representation $\sigma$ of its carrier $X^*$. By definition, $\sigma$ is the partial mapping of $\mathbb{N}^\mathbb{N}$ onto $X^*$ defined by the equivalence
\[\sigma(p)=z\leftrightarrow p\in\mathbb{N}^\mathbb{N}\,\&\,\mathrm{rng}(p)={\mathcal{U}^*}^{-1}(z).\]
We will prove that $\sigma=\varrho^*$. To this end, we will consider the set $Z_3$ of the elements $z$ of $X^*$ such that $\sigma(p)=z\Leftrightarrow\varrho^*(p)=z$ for all $p\in\mathbb{N}^\mathbb{N}$, and we will show that $Z_3=X^*$. If $z\in X$ then ${\mathcal{U}^*}^{-1}(z)=\left\{2i+2\,|\,i\in\mathcal{U}^{-1}(z)\right\}$, hence
\begin{multline*}
\left\{p\in\mathbb{N}^\mathbb{N}\,\left|\,\mathrm{rng}(p)={\mathcal{U}^*}^{-1}(z)\right.\right\}=\left\{\uplambda k.2q(k)+2\,\left|\,q\in\mathbb{N}^\mathbb{N}\,\&\,\mathrm{rng}(q)=\mathcal{U}^{-1}(z)\right.\right\}\\
=\{\uplambda k.2q(k)+2\,|\,\varrho(q)=z\}=\{p\,|\,\varrho^*(p)=z\},
\end{multline*}
thus $z\in Z_3$. The element $\bm{o}$ also belongs to $Z_3$ because ${\mathcal{U}^*}^{-1}(\bm{o})=\{0\}$, hence
\[\left\{p\in\mathbb{N}^\mathbb{N}\,\left|\,\mathrm{rng}(p)={\mathcal{U}^*}^{-1}(\bm{o})\right.\right\}=\{\uplambda k.0\}=\{p\,|\,\varrho^*(p)=\bm{o}\}.\]
Let now $z=(u,v)$ with $u,v\in Z_3$. Then
\[{\mathcal{U}^*}^{-1}(z)=\left\{2J(m,n)+1\,|\,m\in{\mathcal{U}^*}^{-1}(u),\,n\in{\mathcal{U}^*}^{-1}(v)\right\},\]
hence
\begin{multline*}
\left\{p\in\mathbb{N}^\mathbb{N}\,\left|\,\mathrm{rng}(p)={\mathcal{U}^*}^{-1}(z)\right.\right\}\\
=\left\{\uplambda k.2J(q(k),r(k))+1\,\left|\,q,r\in\mathbb{N}^\mathbb{N}\,\&\,\mathrm{rng}(q)
=\mathcal{U}^{-1}(u)\,\&\,\mathrm{rng}(r)=\mathcal{U}^{-1}(v)\right.\right\}\\
=\{\uplambda k.2J(q(k),r(k))+1\,|\,\varrho^*(q)=u\,\&\,\varrho^*(r)=v\}=\{p\,|\,\varrho^*(p)=z\},
\end{multline*}
thus $z\in Z_3$.

Suppose now that the effective topological space $(X,\mathcal{U})$ is computable. Then a recursively enumerable subset $S$ of $\mathbb{N}^3$ exists such that 
\[\mathcal{U}_{i_1}\cap\mathcal{U}_{i_2}=\bigcup\{\mathcal{U}_i\,|\,(i_1,i_2,i)\in S\}\]
for all $i_1,i_2\in\mathbb{N}$. We define $S^*\subseteq\mathbb{N}^3$ in the following inductive way:
\begin{enumerate}[label=(\roman*)]
\item $(2i_1+2,2i_2+2,2i+2)\in S^*$, whenever $(i_1,i_2,i)\in S$.
\item $(0,0,0)\in S^*$.
\item If $(m_1,m_2,m)\in S^*$ and $(n_1,n_2,n)\in S^*$ then
\[(2J(m_1,n_1)+1,2J(m_2,n_2)+1,2J(m,n)+1)\in S^*.\]
\end{enumerate}
The set $S^*$ is recursively enumerable and 
\[\mathcal{U}^*_{i_1}\cap\mathcal{U}^*_{i_2}=\bigcup\{\,\mathcal{U}^*_i\,|\,(i_1,i_2,i)\in S^*\}\]
for all $i_1,i_2\in\mathbb{N}$, hence the effective topological space $(X^*,\mathcal{U}^*)$ is computable.
\end{proof}

\subsection{The case of represented space generated by an effective metric space} 
The notion of effective metric space will be understood as follows: this is a triple $(X,d,\alpha)$, where $(X,d)$ is a metric space, $\alpha$ is a mapping of $\mathbb{N}$ into $X$ and the set $\mathrm{rng}(\alpha)$ is dense in $(X,d)$. The represented space generated by $(X,d,\alpha)$ has as its second term the partial mapping $\varrho$ of $\mathbb{N}^\mathbb{N}$ into $X$ defined as follows: $\varrho(p)=x$ iff $d(\alpha(p(k)),x)<2^{-k}$ for all $k\in\mathbb{N}$. We will define inductively $d^*:X^*\times X^*\to\mathbb{R}$ and $\alpha^*:\mathbb{N}\to X^*$. The function $d^*:X^*\times X^*\to\mathbb{R}$ is defined as follows:
\begin{enumerate}[label=(\roman*)]
\item $d^*(x,x')=\min(d(x,x'),1)$ for all $x,x'\in X$.
\item $d^*(\bm{o},\bm{o})=0$.
\item $d^*(\bm{o},x)=d^*(x,\bm{o})=1$ for all $x\in X$.
\item $d^*(z,z')=d^*(z',z)=1$ whenever $z\in X\cup\{\bm{o}\}$, $z'\in X^*\setminus(X\cup\{\bm{o}\})$.
\item $d^*(z,z')=\max(d^*(L(z),L(z')),d^*(R(z),R(z')))$ for all $z,z'\in X^*\setminus(X\cup\{\bm{o}\})$.
\end{enumerate}
The definition of the mapping $\alpha^*:\mathbb{N}\to X^*$ is the following one:
\begin{enumerate}[label=(\roman*)]
\item $\alpha^*(2i+2)=\alpha(i)$ for all $i\in\mathbb{N}$.
\item $\alpha^*(0)=\bm{o}$.
\item $\alpha^*(2J(m,n)+1)=(\alpha^*(m),\alpha^*(n))$ for all $m,n\in\mathbb{N}$.
\end{enumerate}

\begin{pr}
The triple $(X^*,d^*,\alpha^*)$ is an effective metric space, and $(X^*,\varrho^*)$ is the represented space generated by it. In addition, if the effective metric space $(X,d,\alpha)$ is semicomputable or computable then the effective metric space $(X^*,d^*,\alpha^*)$ is, respectively, also semicomputable or computable.
\end{pr}

\begin{proof}
Let $Z_0$ be the set of the elements $z$ of $X^*$ such that
\begin{equation}\label{basic}
0\le d^*(z,z')=d^*(z',z)\le 1,\ d^*(z,z')=0\Leftrightarrow z=z'
\end{equation}
for all $z'\in X^*$. Obviously $X\cup\{\bm{o}\}\subseteq Z_0$ and $Z_0$ is closed with respect to formation of ordered pairs. Hence $Z_0=X^*$, {\ie} we have \eqref{basic} for all $z,z'\in X^*$. To prove that
\begin{equation}\label{triangle}
d^*(z',z'')\le d^*(z',z)+d^*(z,z'')
\end{equation}
for all $z,z',z''\in X^*$, we consider the set $Z_1$ of the elements $z$ of $X^*$ such that the inequality \eqref{triangle} holds for all $z',z''\in X^*$. We will show that $Z_1=X^*$ by showing that $X\cup\{\bm{o}\}\subseteq Z_1$ and $Z_1$ is closed with respect to formation of ordered pairs. Both statements will be proved by using that the negation of \eqref{triangle} implies the inequalities 
\begin{equation}\label{small}
d^*(z',z)<1,\ d^*(z'',z)<1
\end{equation}
(since $d^*(z',z'')\le 1$, $d^*(z',z)\ge 0$, $d^*(z,z'')\ge 0$). Let $z\in X\cup\{\bm{o}\}$, and suppose the inequality \eqref{triangle} is violated for some $z',z''\in X^*$, hence we have the inequalities \eqref{small} for these $z',z''$. If $z\in X$ these inequalities entail that $z'$ and $z''$ also belong to $X$, hence the negation of \eqref{triangle} turns into the inequality
\[\min(d(z',z''),1)>\min(d(z',z),1)+\min(d(z,z''),1);\]
thus
\begin{gather*}
d(z',z'')>\min(d(z',z),1)+\min(d(z,z''),1),\\
1>\min(d(z',z),1)+\min(d(z,z''),1).
\end{gather*}
Due to the nonnegativity of $\min(d(z',z),1)$ and $\min(d(z,z''),1)$, the second of these inequalities is possible only if
\[\min(d(z',z),1)=d(z',z),\ \min(d(z,z''),1)=d(z,z''),\]
but this, together with the first of them, contradicts the inequality
\[d(z',z'')\le d(z',z)+d(z,z'').\]
If $z=\bm{o}$ then the inequalities \eqref{small} entail $z'=z''=\bm{o}$, and the negation of \eqref{triangle} turns into the false inequality $0>0+0$. Let now $z=(u,v)$, where $u$ and $v$ belong to $Z_1$. We will show that $z$ belongs to $Z_1$ too. Suppose the inequality \eqref{triangle} is violated for some $z',z''\in X^*$, hence we have the inequalities \eqref{small} for these $z',z''$. Since $z\in X^*\setminus(X\cup\{\bm{o}\})$, it follows from these inequalities that $z'$ and $z''$ also belong to $X^*\setminus(X\cup\{\bm{o}\})$, hence $z'=(u',v')$, $z''=(u'',v'')$ for some $u',v',u'',v''\in X^*$. Then
\begin{gather*}
d^*(z',z'')=\max(d^*(u',u''),d^*(v',v'')),\\
d^*(z',z)=\max(d^*(u',u),d^*(v',v)),\\
d^*(z,z'')=\max(d^*(u,u''),d^*(v,v'')),\\
d^*(u',u'')\le d^*(u',u)+d^*(u,u'')\le d^*(z',z)+d^*(z,z''),\\
d^*(v',v'')\le d^*(v',v)+d^*(v,v'')\le d^*(z',z)+d^*(z,z''),
\end{gather*}
hence we may conclude that \eqref{triangle} holds and thus get a contradiction.

After we proved above that $(X^*,d^*)$ is a metric space, we will now show that $\mathrm{rng}(\alpha^*)$ is dense in this metric space. Let $\varepsilon>0$, and let $Z^\varepsilon$ be the set of the elements $z$ of $X^*$ such that $d^*(\alpha^*(k),z)<\varepsilon$ for some $k\in\mathbb{N}$. We will show that $Z^\varepsilon=X^*$ by showing that $X\cup\{\bm{o}\}\subseteq Z^\varepsilon$ and $Z^\varepsilon$ is closed with respect to formation of ordered pairs. Any $z\in X$ belongs to $Z^\varepsilon$ because \[d^*(\alpha^*(2i+2),z)=d^*(\alpha(i),z)=\min(d(\alpha(i),z),1)<\varepsilon\]
for any $i\in\mathbb{N}$ such that $d(\alpha(i),z)<\varepsilon$. The element $\bm{o}$ also belongs to $Z^\varepsilon$ because
\[d^*(\alpha^*(0),\bm{o})=d^*(\bm{o},\bm{o})=0.\]
Suppose now that $z=(u,v)$, where $u$ and $v$ belong to $Z^\varepsilon$. There are $m,n\in\mathbb{N}$ such that $d^*(\alpha^*(m),u)<\varepsilon$ and $d^*(\alpha^*(n),v)<\varepsilon$. Then
\[d^*(\alpha^*(2J(m,n)+1),z)=d^*((\alpha^*(m),\alpha^*(n)),(u,v))=\max(d^*(\alpha^*(m),u),d^*(\alpha^*(n),v))<\varepsilon.\]
Hence $z\in Z^\varepsilon$.

Let now $(X^*,\sigma)$ be the represented space corresponding to the effective metric space $(X^*,d^*,\alpha^*)$, {\ie} $\sigma$ is the partial mapping of $\mathbb{N}^\mathbb{N}$ onto $X^*$ defined by the equivalence
\[\sigma(p)=z\Leftrightarrow\forall k\in\mathbb{N}\left(d^*(\alpha^*(p(k)),z)<2^{-k}\right).\]
We will show that $\sigma=\varrho^*$, where $\varrho$ is the second term of the represented space corresponding to the metric space $(X,d,\alpha)$. Let $Z_2$ be the set of the elements $z$ of $X^*$ such that
\begin{equation}\label{equiv}
\sigma(p)=z\Leftrightarrow\varrho^*(p)=z
\end{equation}
for any $p\in\mathbb{N}^\mathbb{N}$. We will show that $Z_2=X^*$ by showing that $X\cup\{\bm{o}\}\subseteq Z_2$ and $Z_2$ is closed with respect to formation of ordered pairs. Let $z\in X$. Then the inequality $d^*(z',z)<2^{-k}$, where $z'\in X^*$, $k\in\mathbb{N}$, can be satisfied only in the case of $z'\in X$ (since $2^{-k}\le 1$). Therefore the equality $\sigma(p)=z$ can be satisfied only if $\alpha^*(p(k))\in X$ for all $k\in\mathbb{N}$, {\ie} if all $p(k)$ are positive even numbers. In such a case,
\[d^*(\alpha^*(p(k)),z)=d^*(\alpha(p(k)/2-1),z)=\min(d(\alpha(p(k)/2-1),z),1),\]
hence
\[d^*(\alpha^*(p(k)),z)<2^{-k}\Leftrightarrow d(\alpha(p(k)/2-1),z)<2^{-k}.\]
Therefore the equality $\sigma(p)=z$ holds iff $p(k)$ is a positive even number for any $k\in\mathbb{N}$ and
\begin{equation}\label{cauchy}
\forall k\in\mathbb{N}\left(d(\alpha(p(k)/2-1),z)<2^{-k})\right).
\end{equation}
On the other hand, the equality $\varrho^*(p)=z$ holds iff $p(k)$ is a positive even number for any $k\in\mathbb{N}$ and $\varrho(\uplambda k.p(k)/2-1)=z$. By the definition of $\varrho$, the last equality is equivalent to the condition \eqref{cauchy}, and thus we see that $z\in Z_2$.
The element $\bm{o}$ also belongs to $Z_2$, because in the case $z=\bm{o}$ each of the sides of the equivalence \eqref{equiv} is satisfied iff $p(k)=0$ for all $k\in\mathbb{N}$. Suppose now that $z=(u,v)$, where $u,v\in Z_2$. Then the inequality $d^*(z',z)<2^{-k}$, where $z'\in X^*$, $k\in\mathbb{N}$, can be satisfied only in the case of $z'\in X^*\setminus(X\cup\{\bm{o}\})$. Therefore the equality $\sigma(p)=z$ can be satisfied only if $\alpha^*(p(k))\in X^*\setminus(X\cup\{\bm{o}\})$ for any $k\in\mathbb{N}$, {\ie} if all $p(k)$ are positive odd numbers. Any positive odd number has a unique representation in the form $2J(m,n)+1$ with $m,n\in\mathbb{N}$, hence any function in $\mathbb{N}^\mathbb{N}$ with positive odd values can be uniquely represented in the form
\begin{equation}\label{seqofodd}
\uplambda k.2J(q(k),r(k))+1
\end{equation}
with $q,r\in\mathbb{N}^\mathbb{N}$. For any two such $q,r$ and any $k\in\mathbb{N}$, we have
\[d^*(\alpha^*(2J(q(k),r(k))+1),z)=\max(d^*(\alpha^*(q(k)),u),d^*(\alpha^*(r(k)),v)).\]
Therefore
\[d^*(\alpha^*(2J(q(k),r(k))+1),z)<2^{-k}\Leftrightarrow d^*(\alpha^*(p(k)),u)<2^{-k}\ \&\ d^*(\alpha^*(q(k)),v))<2^{-k}\]
for any $k\in\mathbb{N}$, hence
\begin{multline*}
\sigma(\uplambda k.2J(q(k),r(k))+1)=z\Leftrightarrow\sigma(q)=u\ \&\ \sigma(r)=v\\
\Leftrightarrow\varrho^*(q)=u\ \&\ \varrho^*(r)=v\Leftrightarrow\varrho^*(\uplambda k.2J(q(k),r(k))+1)=z.
\end{multline*}
Since the value of $\varrho^*$ can be equal to $z$ only at functions of the form \eqref{seqofodd}, it is thus shown that $z\in Z_2$.

Suppose now $(X,d,\alpha)$ is semicomputable. Then the set 
\[H=\left\{(i,j,k,l)\in\mathbb{N}^4\,\bigg|\,d(\alpha(i),\alpha(j))<\frac{k+1}{l+1}\right\}\]
is recursively enumerable. We will show that $(X^*,d^*,\alpha^*)$ is semicomputable by proving the recursive enumerability of the set
\[H^*=\left\{(i,j,k,l)\in\mathbb{N}^4\,\bigg|\,d^*(\alpha^*(i),\alpha^*(j))<\frac{k+1}{l+1}\right\}.\]
To this end, it is sufficient to indicate that $H^*$ can be defined inductively as follows (with $i,j,k,l,m,m',n,n'$ ranging over $\mathbb{N}$):
\begin{enumerate}[label=(\roman*)]
\item If $k>l$ then $(i,j,k,l)\in H^*$.
\item If $(i,j,k,l)\in H$ then $(2i+2,2j+2,k,l)\in H^*$.
\item $(0,0,k,l)\in H^*$.
\item If $(m,m',k,l)\in H^*$ and $(n,n',k,l)\in H^*$ then $(2J(m,n)+1,2J(m',n')+1,k,l)\in H^*$.
\end{enumerate}

Suppose $(X,d,\alpha)$ is computable. Then a computable function $\delta:\mathbb{N}^3\to\mathbb{Q}$ exists such that
\[|\delta(i,j,k)-d(\alpha(i),\alpha(j))|<2^{-k}\]
for all $i,j,k\in\mathbb{N}$. We define a function $\delta^*:\mathbb{N}^3\to\mathbb{Q}$ by means of the following inductive definition, where $i,j,k,m,n,m',n'$ range over $\mathbb{N}$:
\begin{enumerate}[label=(\roman*)]
\item $\delta^*(2i+2,2j+2,k)=\delta(i,j,k)$.
\item $\delta^*(0,0,k)=0$.
\item $\delta^*(0,2i+2,k)=\delta^*(2i+2,0,k)=1$.
\item $\delta^*(2i,2j+1,k)=\delta^*(2j+1,2i,k)=1$.
\item $\delta^*(2J(m,n)+1,2J(m',n')+1,k)=\max(\delta^*(m,m',k),\delta^*(n,n',k))$. 
\end{enumerate}
Then $\delta^*$ is a computable function from $\mathbb{N}^3$ to $\mathbb{Q}$ such that
\[|\delta^*(i,j,k)-d^*(\alpha^*(i),\alpha^*(j))|<2^{-k}\]
for all $i,j,k\in\mathbb{N}$, hence $(X^*,d^*,\alpha^*)$ is computable.
\end{proof}

\section{On the TTE computability in the Moschovakis extension \\of a represented space}\label{S:3}

In this section, an arbitrary represented space $(X,\varrho)$ will be supposed to be given. TTE computability in the represented space $(X^*,\varrho^*)$ constructed as in the previous section will be considered.

\subsection{Some general properties}
As well known, the values of the second component of a represented space at computable sequences of natural numbers are called {\em computable} elements of the space. The following three statements are easily verifiable, and they entail the corollary after them:
\begin{pr}
The element $\bm{o}$ is a computable element of $(X^*,\varrho^*)$.  
\end{pr}
\begin{pr}
If $x\in X$ then $x$ is a computable element of $(X^*,\varrho^*)$ iff $x$ is a computable element of $(X,\varrho)$.
\end{pr}
\begin{pr}
If $u$ and $v$ belong to $X^*$ then $(u,v)$ is a computable element of $(X^*,\varrho^*)$ iff $u$ and $v$ are computable elements of $(X^*,\varrho^*)$.
\end{pr} 
\begin{co}
The set of the computable elements of the represented space $(X^*,\varrho^*)$ is the closure under the pairing operation of the set consisting of the element $\bm{o}$ and the computable elements of the represented space $(X,\varrho)$.
\end{co}

If $(X_1,\varrho_1)$ and $(X_2,\varrho_2)$ are represented spaces, then, by definition, a {\em $(\varrho_1,\varrho_2)$-realization} of a partial function $\varphi$ from $X_1$ to $X_2$ is any partial mapping $\gamma$ of $\mathbb{N}^\mathbb{N}$ into $\mathbb{N}^\mathbb{N}$ such that  $p\in\mathrm{dom}(\gamma)$, $\gamma(p)\in\mathrm{dom}(\varrho_2)$ and $\varphi(\varrho_1(p))=\varrho_2(\gamma(p))$ for any $p\in\varrho_1^{-1}(\mathrm{dom}(\varphi))$; the function $\varphi$ is said to be {\em$(\varrho_1,\varrho_2)$-computable} if a computable $(\varrho_1,\varrho_2)$-realization of $\varphi$ exists.\footnote{The computability notion for partial mappings of $\mathbb{N}^\mathbb{N}$ into $\mathbb{N}^\mathbb{N}$ is supposed to be introduced in some of the well-known mutually equivalent ways.} Clearly the last condition is satisfied in the particular case when $\mathrm{rng}(\varphi)$ consists of only one element and it is a computable element of $(X_2,\varrho_2)$. Thus, for instance, the constant functions $T_1$ and $F_1$ which are components of the combinatory space $\mathfrak{S}_1(X)$ are $(\varrho^*,\varrho^*)$-computable (with the same $\varrho^*$ as above). Of course, the condition is satisfied also in the case when $(X_1,\varrho_1)=(X_2,\varrho_2)$ and $\varphi$ is the corresponding identity mapping (or some restriction of it). A little more care is needed for the case of the components $L_1$ and $R_1$ of $\mathfrak{S}_1(X)$.
\begin{pr}\label{lr}
The functions $L_1$ and $R_1$ are $(\varrho^*,\varrho^*)$-computable.
\end{pr}

\begin{proof}
A computable $(\varrho^*,\varrho^*)$-realization $\gamma$ of $L_1$ can be defined by setting $\gamma(p)(k)=g(p(k))$, where the function $g:\mathbb{N}\to\mathbb{N}$ is defined by means of the equalities
\[g(0)=0,\ g(2i+2)=1,\footnote{Note that $1=2J(0,0)+1$ because the surjectivity of $J$ and the inequalities $J(m,n)\ge m$, $J(m,n)\ge n$ imply $J(0,0)=0$.}\ g(2J(m,n)+1)=m.\] For $R_1$, the construction of a computable $(\varrho^*,\varrho^*)$-realization is similar -- with right-hand side $n$ of the third equality in the definition of $g$.
\end{proof}

The identity mapping $\mathrm{id}_X$ of $X$ into $X$ can be regarded also as a function from $X$ to $X^*$, as well as a partial function from $X^*$ to $X$.
\begin{pr}\label{identity}
The mapping $\mathrm{id}_X$ is~a~$(\varrho,\varrho^*)$-computable function from~$X$ to~$X^*$, as well as a~$(\varrho^*,\varrho)$-computable partial function from $X^*$ to $X$.
\end{pr}

\begin{proof}
To define a computable $(\varrho,\varrho^*)$-realization of $\mathrm{id}_X$, we may set $\gamma(p)(k)=2p(k)+2$. A computable $(\varrho^*,\varrho)$-realization of $\mathrm{id}_X$ can be defined by setting $\gamma(p)(k)=\lfloor p(k)/2\rfloor\dotminus 1$.
\end{proof}

\begin{co}
For any partial function from $X$ to $X$, its $(\varrho,\varrho)$-computability is equivalent to any of the following three properties: $(\varrho,\varrho^*)$-computability as a partial function from $X$ to~$X^*$, $(\varrho^*,\varrho)$-computability as a partial function from $X^*$ to $X$, $(\varrho^*,\varrho^*)$-computability as a partial function from $X^*$ to $X^*$.
\end{co}

If $(X_0,\varrho_0)$ is a represented space then a partial function $\varphi$ from $X_0$ to $\mathbb{N}$ is said to be {\em$\varrho_0$-computable} if a computable partial mapping $\gamma$ of $\mathbb{N}^\mathbb{N}$ into $\mathbb{N}$ exists such that $p\in\mathrm{dom}(\gamma)$ and $\varphi(\varrho_0(p))=\gamma(p)$ for all $p\in\varrho_0^{-1}(\mathrm{dom}(\varphi))$; a partial function $\psi$ from $\mathbb{N}$ to $X_0$ is said to be {\em$\varrho_0$-computable} if a computable partial mapping $\delta$ of $\mathbb{N}$ into $\mathbb{N}^\mathbb{N}$ exists such that $n\in\mathrm{dom}(\delta)$, $\delta(n)\in\mathrm{dom}(\varrho_0)$ and $\psi(n)=\varrho_0(\delta(n))$ for all $n\in\mathrm{dom}(\psi)$.

\begin{pr}
A partial function from $X$ to $\mathbb{N}$ is $\varrho$-computable iff it is $\varrho^*$-computable as a partial function from $X^*$ to $\mathbb{N}$.
\end{pr}

\begin{proof}
Follows from Proposition \ref{identity}.
\end{proof}

\begin{pr}
The function $n\mapsto n^*$ from $\mathbb{N}$ to $X^*$ and its inverse function are $\varrho^*$-computable.
\end{pr}

\begin{proof}
Let $f:\mathbb{N}\to\mathbb{N}$ be defined by means of the equalities $f(0)=0$, $f(n+1)=J(0,f(n))$, and, for any $n\in\mathbb{N}$, let $\delta(n):\mathbb{N}\to\mathbb{N}^\mathbb{N}$ be defined by means of the equality $\delta(n)(k)=f(n)$. Then $\delta$ is a computable mapping of $\mathbb{N}$ into $\mathrm{dom}(\varrho^*)$ and $n^*=\varrho^*(\delta(n))$ for all $n\in\mathbb{N}$, hence the function $n\mapsto n^*$ is $\varrho^*$-computable. Let $g$ be the partial function in $\mathbb{N}$ defined by means of the equality $g(m)=\upmu n[f(n)=m]$, and let $\gamma$ be the partial mapping of $\mathbb{N}^\mathbb{N}$ into $\mathbb{N}$ defined by means of the equality $\gamma(p)=g(p(0))$. Then $\gamma$ is computable, and, whenever $p\in\mathrm{dom}(\varrho^*)$ and $\varrho^*(p)=n^*$, where $n\in\mathbb{N}$, the equality $n=\gamma(p)$ holds. This shows the $\varrho^*$-computability of the inverse function of $n\mapsto n^*$.
\end{proof}

\begin{co}
A partial function $\varphi$ from $X^*$ to $\mathbb{N}$ is $\varrho^*$-computable iff the~partial function $z\mapsto\varphi(z)^*$ from $X^*$ to $X^*$ is $(\varrho^*,\varrho^*)$-computable.
\end{co}

\begin{co}
A partial function $h$ from $\mathbb{N}$ to $\mathbb{N}$ is computable iff the~partial function $n^*\mapsto h(n)^*$ from $\mathbb{N}^*$ to $\mathbb{N}^*$ is $(\varrho^*,\varrho^*)$-computable.
\end{co}

\subsection{Closedness of the set of the TTE computable functions from \texorpdfstring{$\mathcal{F}_1$}{F1} under \texorpdfstring{$\mathfrak{S}_1(X)$}{S1}-computability} 
\begin{thm}\label{compF1}
Let $\mathcal{B}_1$ be the set of all functions from $\mathcal{F}_1$ which are $(\varrho^*,\varrho^*)$-computable. Then all functions $\mathfrak{S}_1(X)$-computable in $\mathcal{B}_1$ belong to $\mathcal{B}_1$.
\end{thm}

\begin{proof}
Having in mind footnote \ref{sf} and Proposition \ref{lr}, as well as the fact that $\mathcal{B}_1$ is closed under composition, we must prove that $\mathcal{B}_1$ is closed under the operations $\Pi_1$ and $\mathfrak{S}_1(X)$-iteration. Let $\varphi_1,\varphi_2\in\mathcal{B}_1$. Then there are computable $(\varrho^*,\varrho^*)$-realizations $\gamma_1$ and $\gamma_2$ of $\varphi_1$ and $\varphi_2$, respectively. To show that $\Pi_1(\varphi_1,\varphi_2)\in\mathcal{B}_1$, we make use of the mapping $\gamma:\mathrm{dom}(\gamma_1)\cap\mathrm{dom}(\gamma_2)\to\mathbb{N}^\mathbb{N}$ defined by means of the equality
\[\gamma(p)(k)=2J(\gamma_1(p)(k),\gamma_2(p)(k))+1.\]
Let now $\iota$ be the $\mathfrak{S}_1(X)$-iteration of $\varphi_1$ controlled by $\varphi_2$. The proof that $\iota\in\mathcal{B}_1$ is based on the following observation (where $\gamma_1^0$ must be considered as denoting the identity mapping of $\mathbb{N}^\mathbb{N}$ into $\mathbb{N}^\mathbb{N}$): if $p\in{\varrho^*}^{-1}(\mathrm{dom}(\iota))$ then a natural number $n$ exists such that $p\in\mathrm{dom}(\gamma_1^n)$, $\gamma_1^n(p)\in\mathrm{dom}(\varrho^*)$, $\iota(\varrho^*(p))=\varrho^*(\gamma_1^n(p))$, $\gamma_1^i(p)\in\mathrm{dom}(\gamma_2)$ for all $i\le n$ and 
\[\gamma_2(\gamma_1^i(p))(0)\!\!\!\!\mod 2=
\begin{cases}
0&\text{if }i=n,\\
1&\text{if }i<n.
\end{cases}\]
The mappings $\gamma_1$ and $\gamma_2$ can be extended to computable operators $\Gamma_1$ and $\Gamma_2$ in the set of all unary partial functions in $\mathbb{N}$. If $p\in{\varrho^*}^{-1}(\mathrm{dom}(\iota))$ then 
\[\iota(\varrho^*(p))=\varrho^*(\Gamma_1^n(p))\text{ with }n=\upmu i[\Gamma_2(\Gamma_1^i(p))(0)\!\!\!\!\mod 2=0].\]
Let $\Gamma_1^\dag$ be the operator transforming unary partial functions in $\mathbb{N}$ into binary ones which is defined by the equality
\[\Gamma_1^\dag(f)(i,k)=\Gamma_1^i(f)(k).\]
Let the operators $\Gamma_2^\dag$ and $\Gamma$ transforming unary partial functions in $\mathbb{N}$ into unary ones be defined as follows:
\begin{gather*}
\Gamma_2^\dag(f)(i)=\Gamma_2(\Gamma_1^i(f))(0),\\
\Gamma(f)(k)=\Gamma_1^\dag(f)(\upmu i[\Gamma_2^\dag(f)(i)\!\!\!\!\mod 2=0],k).
\end{gather*}
One consecutively shows that $\Gamma_1^\dag,\Gamma_2^\dag,\Gamma$ are computable (this can be done, for instance, by using their continuity and the effectiveness of their restrictions to partial recursive functions). For all $p\in{\varrho^*}^{-1}(\mathrm{dom}(\iota))$, $\Gamma(p)\in\mathrm{dom}(\varrho^*)$ and the equality $\iota(\varrho^*(p))=\varrho^*(\Gamma(p))$ holds. Therefore some restriction of $\Gamma$ will be a realization of $\iota$.
\end{proof}

\begin{co}
Let $\mathcal{B}_1$ be the same as in Theorem \ref{compF1}, and let $\Gamma_1,\ldots,\Gamma_l$ be mappings of $\mathcal{F}_1^l$ into $\mathcal{F}_1$ which are $\mathfrak{S}_1(X)$-computable in $\mathcal{B}_1$. Then the components of the least solution of the system of inequalities \eqref{frt} belong to $\mathcal{B}_1$.
\end{co}

\begin{ex}[concerning computability of real numbers and real functions in the usual sense]\label{rec1}
Let $c$ be a computable real number. Let $\alpha:(c,+\infty)\to\mathbb{R}$ and $\beta:(-\infty,c)\times\mathbb{R}\to\mathbb{R}$ be computable. Suppose $\varphi:\mathbb{R}\setminus\{c,c-1,c-2,c-3,\ldots\}\to\mathbb{R}$ is defined recursively as follows:
\[\varphi(x)=
\begin{cases}
\alpha(x)&\text{if }x>c,\\
\beta(x,\varphi(x+1))&\text{if }x<c,\ x\not\in\{c-1,c-2,c-3,\ldots\}.
\end{cases}\]
Then $\varphi$ is also computable. To prove this we may consider the combinatory space $\mathfrak{S}_1(X)$ corresponding to $X=\mathbb{R}$ and use the fact that $\varphi$ is the only solution in it of the equation
\[\theta=\Sigma_1(\chi,\alpha,\beta\Pi_1(I_1,\theta\sigma)),\]
where $\sigma:\mathbb{R}\to\mathbb{R}$ and $\chi:\mathbb{R}\setminus\{c\}\to\{\bm{o},(\bm{o},\bm{o})\}$ are defined as follows:
\begin{align*}
\sigma(x)&=x+1,\\
\chi(x)&=
\begin{cases}
(\bm{o},\bm{o})&\text{if }x>c,\\
\bm{o}&\text{if }x<c
\end{cases}
\end{align*}
(to see that the mapping $\theta\mapsto\Sigma_1(\chi,\alpha,\beta\Pi_1(I_1,\theta\sigma))$ is $\mathcal{B}_1$-computable, we make use of footnote \ref{compSigma}).
\end{ex}

\begin{re}\label{Sigma1-closed}
By footnote \ref{sf}, the operation $\Sigma_1$ preserves $\mathfrak{S}_1(X)$-computability in $\mathcal{B}_1$, hence, by Theorem~\ref{compF1}, the set $\mathcal{B}_1$ is closed under $\Sigma_1$. It is easy to show this also directly: whenever $\gamma_1,\gamma_2,\gamma_3$ are $(\varrho^*,\varrho^*)$-realizations, respectively, of $\varphi_1,\varphi_2,\varphi_3\in\mathcal{F}_1$, we can construct a $(\varrho^*,\varrho^*)$-realization $\gamma$ of $\Sigma_1(\varphi_1,\varphi_2,\varphi_3)$ by setting
\[\gamma(p)(k)=
\begin{cases}
\gamma_2(p)(k)&\text{if }\gamma_1(p)(0)\text{ is odd,}\\
\gamma_3(p)(k)&\text{otherwise}
\end{cases}\]
for all $p\in\mathrm{dom}(\gamma_1)$ such that either $\gamma_1(p)(0)$ is odd and $p\in\mathrm{dom}(\gamma_2)$, or $\gamma_1(p)(0)$ is even and $p\in\mathrm{dom}(\gamma_3)$; if $\gamma_1,\gamma_2,\gamma_3$ are computable then $\gamma$ is also computable.
\end{re}

\subsection{Closedness of the set of the TTE computable functions from \texorpdfstring{$\mathcal{F}_2$}{F2} under \texorpdfstring{$\mathfrak{S}_2(X)$}{S2}-computability} 
Theorem \ref{compF1} has an analog for the $(\varrho^*,\varrho^*)$-computability of functions from $\mathcal{F}_2$ provided that the following definition is accepted (compare with \cite[Clause 4 of Definition 3.1.3]{weihrauch:ca}): if $(X_1,\varrho_1)$ and $(X_2,\varrho_2)$ are represented spaces, then a {\em $(\varrho_1,\varrho_2)$-realization} of a partial function $\varphi$ from $X_1$ to the set of the nonempty subsets of $X_2$ is any partial mapping $\gamma$ of $\mathbb{N}^\mathbb{N}$ into $\mathbb{N}^\mathbb{N}$ such that $p\in\mathrm{dom}(\gamma)$, $\gamma(p)\in\mathrm{dom}(\varrho_2)$ and $\varrho_2(\gamma(p))\in\varphi(\varrho_1(p))$ for any $p\in\varrho_1^{-1}(\mathrm{dom}(\varphi))$; the function $\varphi$ is said to be {\em$(\varrho_1,\varrho_2)$-computable} if a computable $(\varrho_1,\varrho_2)$-realization of $\varphi$ exists. 

\begin{thm}\label{compF2}
Let $\mathcal{B}_2$ be the set of all functions from $\mathcal{F}_2$ which are $(\varrho^*,\varrho^*)$-computable. Then all functions $\mathfrak{S}_2(X)$-computable in $\mathcal{B}_2$ belong to $\mathcal{B}_2$.
\end{thm}

\begin{proof}
For any function $\varphi\in\mathcal{F}_1$, the corresponding function $\tilde{\varphi}$ belongs to $\mathcal{B}_2$ iff $\varphi$ belongs to~$\mathcal{B}_1$. Hence $L_2$ and $R_2$ belong to $\mathcal{B}_2$. The set $\mathcal{B}_2$ is obviously closed under $\mathfrak{S}_2(X)$-multiplication.\footnote{Note that $\mathcal{B}_2$ is surely not closed with respect to the usual composition. Here is an example showing this. Let $h$ be a non-computable function from $\mathbb{N}$ to $\mathbb{N}\setminus\{0\}$, and let $\varphi_1,\varphi_2\in\mathcal{F}_2$ be defined as follows: $\mathrm{dom}(\varphi_1)=\mathbb{N}^*\setminus\{\bm{o}\}$, $\varphi_1(x)=\{x\}$ for any $x\in\mathbb{N}^*\setminus\{\bm{o}\}$, $\mathrm{dom}(\varphi_2)=\mathbb{N}^*$ and, for any $n\in\mathbb{N}$, $\varphi_2(n^*)=\{\bm{o},h(n)^*\}$. Then $\varphi_1,\varphi_2\in\mathcal{B}_2$, but the usual composition $\varphi_1\varphi_2$ transforms $n^*$ into $\{h(n)^*\}$ for any $n\in\mathbb{N}$, therefore it does not belong to $\mathcal{B}_2$.} Hence it remains to prove that $\mathcal{B}_2$ is closed under the operations $\Pi_2$ and $\mathfrak{S}_2(X)$-iteration. Let $\varphi_1,\varphi_2\in\mathcal{B}_2$.  Then there are computable $(\varrho^*,\varrho^*)$-realizations $\gamma_1$ and $\gamma_2$ of $\varphi_1$ and $\varphi_2$, respectively.  The fact that the function $\Pi_2(\varphi_1,\varphi_2)$ belongs to $\mathcal{B}_2$ is shown by using its $(\varrho^*,\varrho^*)$-realization defined as in the proof of Theorem \ref{compF1}. Let now $\iota$ be the $\mathfrak{S}_2(X)$-iteration of $\varphi_1$ controlled by~$\varphi_2$. In order to prove that $\iota$ also belongs to $\mathcal{B}_2$, we will show that the observation used in the mentioned proof holds in the present situation with the following modification: the equality $\iota(\varrho^*(p))=\varrho^*(\gamma_1^n(p))$ must be replaced by the condition $\varrho^*(\gamma_1^n(p))\in\iota(\varrho^*(p))$. Let $p\in{\varrho^*}^{-1}(\mathrm{dom}(\iota))$, and let $H_p$ be the set of all natural numbers $n$ such that 
\begin{enumerate}[label=(\roman*)]
\item $p\in\mathrm{dom}(\gamma_1^n)$;
\item $\gamma_1^i(p)\in\mathrm{dom}(\gamma_2)$ and $\gamma_2(\gamma_1^i(p))(0)$ is odd for all $i<n$.
\end{enumerate}
Clearly $0\in H_p$, and all natural numbers which are less than a number from $H_p$ also belong to~$H_p$. Let $\dot{H}_p$ be the set of the numbers $n\in H_p$ such that $\gamma_1^i(p)\in\mathrm{dom}(\varrho^*)$ for all $i\le n$ and the sequence
\begin{equation}\label{path}
\varrho^*(p),\varrho^*(\gamma_1(p)),\ldots,\varrho^*(\gamma_1^n(p))
\end{equation}
is a $\varphi_1,\varphi_2$-path. We will show that actually $\dot{H}_p=H_p$, {\ie} that $n\in H_p$ implies $n\in\dot{H}_p$. This will be done by induction. The implication holds for $n=0$ since the statement that $0\in\dot{H}_p$ is trivially true. Before taking the inductive step, we will show that
\begin{gather}
\varrho^*(\gamma_1^n(p))\in\mathrm{dom}(\varphi_2)\,\&\,(\varphi_2(\varrho^*(\gamma_1^n(p)))\setminus(X\cup\{\bm{o}\})\ne\varnothing\Rightarrow\varrho^*(\gamma_1^n(p))\in\mathrm{dom}(\varphi_1)),\label{reg}\\
\gamma_1^n(p)\in\mathrm{dom}(\gamma_2),\ \gamma_2(\gamma_1^n(p))\in\mathrm{dom}(\varrho^*),\  \varrho^*(\gamma_2(\gamma_1^n(p)))\in\varphi_2(\varrho^*(\gamma_1^n(p)))\label{chk}
\end{gather}
for any $n\in\dot{H}_p$. Suppose $n\in\dot{H}_p$. Then $\varrho^*(p),\varrho^*(\gamma_1(p)),\ldots,\varrho^*(\gamma_1^n(p))$ is a $\varphi_1,\varphi_2$-path whose initial term belongs to $\mathrm{dom}(\iota)$, therefore the conjunction \eqref{regular} holds with $\varrho^*(\gamma_1^n(p)),\varphi_1,\varphi_2$ in the roles of $u,\sigma,\chi$, respectively, {\ie} the conjunction \eqref{reg} holds, and the first term of this conjunction implies \eqref{chk}. Let now $n$ be a natural number such that $n\in H_p$ implies $n\in\dot{H}_p$, and suppose that $n+1\in H_p$. Then $n\in H_p$, hence $n\in\dot{H}_p$ and therefore the statements \eqref{chk} hold. The assumption $n+1\in H_p$ entails that $\gamma_2(\gamma_1^n(p))(0)$ is odd, hence $\varrho^*(\gamma_2(\gamma_1^n(p)))\not\in X\cup\{\bm{o}\}$, and therefore $\varphi_2(\varrho^*(\gamma_1^n(p)))\setminus(X\cup\{\bm{o}\})\ne\varnothing$. The second term of the conjunction \eqref{reg} implies that $\varrho^*(\gamma_1^n(p))\in\mathrm{dom}(\varphi_1)$, and it follows from here that $\gamma_1^{n+1}(p))
\in\mathrm{dom}(\varrho^*)$, $\varrho^*(\gamma_1^{n+1}(p))\in\varphi_1(\varrho^*(\gamma_1^n(p))$. Thus $n+1\in\dot{H}_p$.

It is easy to see now that $H_p$ is a finite set -- otherwise there would exist an infinite $\varphi_1,\varphi_2$-path $\varrho^*(p),\varrho^*(\gamma_1(p)),\varrho^*(\gamma_1^2(p)),\ldots$, and this is impossible because $\varrho^*(p)\in\mathrm{dom}(\iota)$. Let $n$ be the maximal number in $H_p$. Then the sequence \eqref{path} is a $\varphi_1,\varphi_2$-path. Making use of \eqref{reg} and \eqref{chk}, we will show that $\gamma_2(\gamma_1^n(p))(0)$ is even and $\varrho^*(\gamma_1^n(p))\in\iota(\varrho^*(p))$. Suppose $\gamma_2(\gamma_1^n(p))(0)$ is odd. Then $\varrho^*(\gamma_2(\gamma_1^n(p)))\not\in X\cup\{\bm{o}\}$, therefore $\varrho^*(\gamma_1^n(p))\in\mathrm{dom}(\varphi_1)$. It follows from here that $\gamma_1^n(p)\in\mathrm{dom}(\gamma_1)$ and $\varrho^*(\gamma_1^{n+1}(p))\in\varphi_1(\varrho^*(\gamma_1^n(p)))$. Thus the sequence
\[\varrho^*(p),\varrho^*(\gamma_1(p)),\ldots,\varrho^*(\gamma_1^n(p)),\varrho^*(\gamma_1^{n+1}(p))\]
turns out to be also a $\varphi_1,\varphi_2$-path, contrary to the maximality of $n$. The fact that $\gamma_2(\gamma_1^n(p))(0)$ is even entails $\varrho^*(\gamma_2(\gamma_1^n(p)))\in X\cup\{\bm{o}\}$, hence $\varrho^*(\gamma_1^n(p))\in\iota(\varrho^*(p))$.

Let $\Gamma_1$, $\Gamma_2$, $\Gamma_1^\dag$, $\Gamma_2^\dag$ and $\Gamma$ be constructed in the same way as in the proof of Theorem~\ref{compF1}. Then $\Gamma(p)\in\mathrm{dom}(\varrho^*)$ and $\varrho^*(\Gamma(p))\in\iota(\varrho^*(p))$, whenever $p\in{\varrho^*}^{-1}(\mathrm{dom}(\iota))$.
\end{proof}

\begin{co}
Let $\mathcal{B}_2$ be the same as in Theorem \ref{compF2}, and let $\Gamma_1,\ldots,\Gamma_l$ be mappings of $\mathcal{F}_2^l$ into $\mathcal{F}_2$ which are $\mathfrak{S}_2(X)$-computable in $\mathcal{B}_2$. Then the components of the least solution of the system of inequalities \eqref{frt} belong to $\mathcal{B}_2$.
\end{co}

\begin{ex}[concerning computability of real numbers and real functions in the usual sense]\label{rec2}
Let $c$ be a computable real number. Let the real-valued functions $\alpha$ and $\beta$ with domains $[c,+\infty)$ and $(-\infty,c+1)\times\mathbb{R}$, respectively, be computable and satisfy the condition 
\[\alpha(x)=\beta(x,\alpha(x+1))\]
for $c\le x<c+1$. Suppose the function $\varphi:\mathbb{R}\to\mathbb{R}$ is defined recursively as follows:
\[\varphi(x)=
\begin{cases}
\alpha(x)&\text{if }x\ge c,\\
\beta(x,\varphi(x+1))&\text{otherwise.}
\end{cases}\]
Then $\varphi$ is also computable. To prove this, we may consider the combinatory space $\mathfrak{S}_2(X)$ corresponding to $X=\mathbb{R}$ and use the fact that $\tilde{\varphi}$ is the only solution in it of the equation
\[\theta=\Sigma_2(\chi,\tilde{\alpha},\tilde{\beta}\Pi_2(I_2,\theta\tilde{\sigma})),\]
where $\sigma$ is the same as in Example \ref{rec1}, and $\chi\in\mathcal{F}_2$ with domain $\mathbb{R}$ is defined as follows:
\[\chi(x)=
\begin{cases}
\{(\bm{o},\bm{o})\}&\text{if }x\ge c+1,\\
\{\bm{o},(\bm{o},\bm{o})\}&\text{if }c<x<c+1,\\
\{\bm{o}\}&\text{if }x\le c.
\end{cases}\]
\end{ex}

\begin{re}
The statement that the function $\varphi$ considered in the above example is computable remains valid if we replace the assumptions about the domain of $\beta$ and about the interconnection between $\alpha$ and $\beta$ with the following ones:
\[\mathrm{dom}(\beta)\supseteq(-\infty,c]\times\mathbb{R},\ \ \alpha(c)=\beta(c,\alpha(c+1)).\]
Indeed, let us then consider the function $\alpha_1:[c-1,c]\to\mathbb{R}$ defined by means of the equality $\alpha_1(x)=\beta(x,\alpha(x+1))$. This function is also computable, and the equality $\alpha_1(c)=\alpha(c)$ holds. Let $\bar{\alpha}$ be the join of $\alpha$ and $\alpha_1$, {\ie} their common extension with domain $[c-1,+\infty)$. By \cite[Lemma~4.3.5]{weihrauch:ca}, the function $\bar{\alpha}$ is computable too.\footnote{We may also follow Alex Simpson's proof (mentioned in \cite[Subsection~1.1]{escardo:easdbc}) of the needed instance of the lemma in question and derive the computability of $\bar{\alpha}$ directly from the equality
\[\bar{\alpha}(x)=\alpha(\max(x,c))+\alpha_1(\min(x,c))-\alpha(c).\]} For $c-1\le x<c$, the equalities 
\[\bar{\alpha}(x)=\beta(x,\alpha(x+1))=\beta(x,\varphi(x+1))=\varphi(x)\]
hold, therefore we have
\[\varphi(x)=
\begin{cases}
\bar{\alpha}(x)&\text{if }x\ge c-1,\\
\beta(x,\varphi(x+1))&\text{otherwise.}
\end{cases}\]
for all $x\in\mathbb{R}$. Then we may get the needed conclusion by applying the result from Example~\ref{rec2} with $c-1$, $\bar{\alpha}$ and the restriction of $\beta$ to $(-\infty,c)\times\mathbb{R}$ in the roles of $c$, $\alpha$ and~$\beta$, respectively. 
\end{re}

\begin{re}
In view of footnote \ref{sf}, Theorem \ref{compF2} implies that $\mathcal{B}_2$ is closed with respect to the operation $\Sigma_2$. A direct proof of the last statement can be given similarly to what was done in Remark~\ref{Sigma1-closed}: if $\gamma_1,\gamma_2,\gamma_3$ are $(\varrho^*,\varrho^*)$-realizations, respectively, of $\varphi_1,\varphi_2,\varphi_3\in\mathcal{F}_2$ then the corresponding mapping $\gamma$ constructed in the same way as there turns out to be a $(\varrho^*,\varrho^*)$-realization of $\Sigma_2(\varphi_1,\varphi_2,\varphi_3)$.
\end{re}

\section*{Acknowledgment}
The author thanks an anonymous referee for certain helpful suggestions and corrections.

\end{document}